\documentclass[a4paper,11pt]{amsart}
\usepackage{amssymb}
\usepackage[english]{babel}

\newcommand{\cE}{{\mathcal{E}}}
\newcommand{\cU}{{\mathcal{U}}}
\newcommand{\cO}{{\mathcal{O}}}

\newcommand{\cI}{{\mathcal{I}}}
\newcommand{\cL}{{\mathcal{L}}}
\newcommand{\cP}{{\mathcal{P}}}

\newcommand{\PP}{{\mathbf{P}}}
\newcommand{\ZZ}{{\mathbf{Z}}}
\newcommand{\CC}{{\mathbf{C}}}

\newcommand{\QQ}{{\mathbf{Q}}}

\newcommand{\cQ}{{\mathcal{Q}}}

\newcommand{\str}{{\mathcal{O}}}
\newcommand{\cZ}{\mathcal{Z}}

\newcommand{\cT}{{\mathcal{T}}}

\renewcommand{\phi}{\varphi}
    \newtheorem{lemma}{Lemma}[section]
    \newtheorem{proposition}[lemma]{Proposition}
    \newtheorem{theorem}[lemma]{Theorem}
    
    \newtheorem{corollary}[lemma]{Corollary}
    \newtheorem{procedure}[lemma]{Procedure}

   \theoremstyle{definition}
    \newtheorem{definition}[lemma]{Definition}
    
    \newtheorem{example}[lemma]{Example}
    \newtheorem{remark}[lemma]{Remark}
    \newtheorem{construction}[lemma]{Construction}
    \DeclareMathOperator{\rank}{rank}

\DeclareMathOperator{\sing}{{sing}}

\DeclareMathOperator{\codim}{{codim}}

\DeclareMathOperator{\prim}{{prim}}

\DeclareMathOperator{\red}{{red}}

\DeclareMathOperator{\MW}{{MW}}

\DeclareMathOperator{\coker}{{coker}}

\DeclareMathOperator{\Gr}{{Gr}}
\DeclareMathOperator{\spa}{{span}}
\DeclareMathOperator{\CCon}{{Con}}
\DeclareMathOperator{\Bl}{{Bl}}
\DeclareMathOperator{\ord}{{ord}}
\DeclareMathOperator{\sm}{{smooth}}

\usepackage[all]{xy}
\begin{document}
\title[Elliptic threefolds]{Calculating the Mordell-Weil rank of elliptic
threefolds and the cohomology of singular hypersurfaces}
\author{Klaus Hulek}
\address{Institut f\"ur Al\-ge\-bra\-ische Geo\-me\-trie, Leibniz Universit\"at
Hannover, Wel\-fen\-gar\-ten 1, 30167 Hannover, Germany}
\email{hulek@math.uni-hannover.de}
\author{Remke Kloosterman}
\address{Institut f\"ur Mathematik, Humboldt Universit\"at zu Berlin, Unter den Linden 6, 10099 Berlin, Germany}
\email{klooster@math.hu-berlin.de}

\begin{abstract}
In this paper we give a method for calculating the rank of a general
elliptic curve over the field of rational functions in two variables. 
We reduce this problem to calculating the
cohomology of a singular hypersurface in a weighted projective $4$-space. We
then
give a method for calculating the cohomology of a certain class of singular
hypersurfaces, extending work of Dimca for the isolated singularity case. 
\end{abstract}
\subjclass{14J30 (primary); 14J70,  32S20, 32S35, 32S50 (secondary)}
\keywords{}
\date{\today}
\thanks{The authors wish to thank Chris Peters and Joseph Steenbrink for giving
us a preview of their upcoming book \cite{PS} and  Eduard
Looijenga and Orsola Tommasi for providing results from algebraic topology.
We wish to thank Noriko Yui for drawing our attention
to the examples of Hirzebruch discussed in Section~\ref{secBeh}.}
\maketitle

\section{Introduction}\label{secInt}

Throughout this paper we work over the field of complex numbers $\CC$.
We study families $\pi:X \to S$ of elliptic curves over rational
surfaces, i.e., $X$ is a smooth threefold, $S$ a smooth rational surface and
$\pi$ is a flat morphism admitting a section $\sigma_0:S\to X$. Throughout this 
paper we will  assume that  $X$ is not birational to a product $E\times S'$
with 
$E$ an elliptic curve and $S'$ a rational surface.

The two main invariants of $\pi$ are its configuration of singular fibers and
the Mordell-Weil group $\MW(\pi)$ consisting of {\it rational} sections of
$\pi$.
Unlike the configuration of singular fibers the Mordell-Weil group is a
birational invariant (in the sense of Section~\ref{secSetup}).

The configuration of singular fibers is well-understood. The general fiber of
$\pi$ is an elliptic curve over $\CC(S)$, in particular we have an equation of
the form
 \begin{equation}\label{Weieqn} y^2=x^3+Ax+B, \mbox{ where } A,B\in
\CC(S).\end{equation}
The singular fibers lie over the curve $\Delta$ given by the zero and pole
divisor of $4A^3+27B^2$. The fiber-type over a general point $p$ of some
irreducible component of $\Delta$ can be easily calculated using Tate's
algorithm. The fiber-type over a special point can be calculated using the work
of Miranda \cite{MirEllThree}.

In this paper we concentrate on the Mordell-Weil group $\MW(\pi)$.  Using the
Shioda-Tate-Wazir  formula \cite[Theorem 4.2]{Waz} one can relate 
the rank of $\MW(\pi)$ to the Picard numbers $\rho(S)$ and $\rho(X)$ 
and the type  of singular fibers
of $\pi$ over a general point of each component of $\Delta$. In  general it
turns out to be rather hard to calculate $\rho(X)$ directly. Even in the case of
elliptic surfaces it is a difficult problem to calculate $\rho(X)$ for a given
example,
this can only be done in very specific cases, see e.g. \cite{KloR15}.

The main idea is the following: 
every elliptic threefold over a rational surface (with a section) has a model as
a hypersurface $Y$ of degree $6n$ in the weighted projective space
$\PP:=\PP(2n,3n,1,1,1)$, for some $n$. The existence of such a model (with
minimal $n$) is a direct
consequence of the existence of a (global minimal) Weierstrass equation for an
elliptic
curve over the function field $\CC(S)$ of $S$. Whenever we refer to a minimal
model in this 
paper we mean the model given by a minimal Weierstrass equation, not to a
minimal model in the 
sense of Mori theory. 
In general, this threefold $Y$ is singular. In the first part of this paper we
show 

\begin{theorem}\label{thmCohRel} Let $\pi: X \to S$ be an elliptic threefold $X$
over a rational surface $S$ and let $Y$ be a minimal model of $X/S$ in
$\PP(2n,3n,1,1,1)$. Assume that $H^4(Y,\QQ)$ has a pure weight 4 Hodge
structure.
Then
 \[ \rank \MW(\pi) = \rank (H^{2,2}( H^4(Y,\CC)) \cap H^4(Y,\ZZ)) - 1.\]
\end{theorem}
One can easily show that the rank of the image of the cycle class map in $H^4(Y,\ZZ)$ is at 
least $1+\rank \MW(\pi)+1$. Hence  it follows from this theorem that a multiple of a Hodge class is algebraic.

The advantage of this theorem is that we can relate the computation of
$\MW(\pi)$ 
to a computation for a hypersurface in weighted projective space. The latter
problem 
is indeed doable as we will show in the second part of the paper.

The assumption that $H^4(Y,\QQ)$ has a pure weight 4 Hodge structure is very
weak.
We do not know of examples such that $H^4(Y,\QQ)$ does not  have a pure
weight
4 Hodge structure. Later on we will describe a large class of elliptic
threefolds for which we have a method to calculate $H^4(Y,\QQ)$. Each member
$Y$ of this class has a pure weight 4 Hodge structure on $H^4(Y,\QQ)$. 

For a complete proof we refer to Section~\ref{secComMW}. Here we only give a
sketch of the
proof: from \cite{MirEllThree} we get a factorization of the birational map
$Y\dasharrow X$. This factorization is sufficiently explicit to relate the
difference
$\rho(X)-\rho(S)$ to $H^{2,2}(H^4(Y,\CC)) \cap H^{4}(Y,\ZZ)$. The configuration
of singular fibers of $\pi$ is relatively easy to compute. Applying the
Shioda-Tate-Wazir formula then yields the proof.

If $X$ is chosen sufficiently general then $Y$ is quasismooth and hence a
$V$-manifold. 
Using this one can show that
$h^4(Y)=1$.
Theorem~\ref{thmCohRel} then implies $\rank \MW(\pi)=0$. For this reason we
shall 
focus in this paper on non-quasismooth hypersurfaces.
 
A more explicit form of the above remark is the following (see Corollary
\ref{corTrvMWTxt}):
 \begin{corollary}\label{corTrvMW} Let $\pi: X \to S$ be an elliptic threefold
associated with a hypersurface
defined by \[ y^2=x^3+Px+Q  \]
 with $P\in \CC[z_0,z_1,z_2]_{4n}$ and $Q\in \CC[z_0,z_1,z_2]_{6n}$, such that 
 \begin{enumerate}
 \item the curve $\Delta:4P^3+27Q^2=0$ is reduced with only
double points as singularities and $Q$ vanishes at each of these double points
or 
 \item  $P$ is identically zero  and  $Q=0$ defines a smooth curve of degree
$6n$
in $\PP^2$.
 \end{enumerate} 
Then $\rank \MW(\pi)=0$.
 \end{corollary}

 Theorem~\ref{thmCohRel} implies the following  two results:
 if we call $\delta=h^4(Y)-1$ the defect of $Y$ then $\rank \MW(\pi)\leq
\delta$.  (The notion of defect for 
singular hypersurfaces is due to Clemens \cite{CleDS}.)
Moreover, one can show that  $\MW(\pi)\otimes \QQ$   is isomorphic to the group
of Weil Divisors on $Y$ 
modulo the Cartier Divisors tensored with $\QQ$. 
 
In the case of elliptic surfaces
$\psi: E\to \PP^1$ one has a theorem similar to Theorem~\ref{thmCohRel}.
However, we are not aware of any statement concerning elliptic surfaces similar
to  Corollary~\ref{corTrvMW}.
The reason for this is the following: let $T$ be a surface in weighted
projective space corresponding to $\psi$. The degree of $T$ is divisible by 6.
Set
$n=\deg(T)/6$. 
 One can show that $\rank \MW(\psi)= \rank (H^{1,1}(H^2(T,\CC)) \cap
H^2(T,\ZZ))-1$ and
$h^{2,0}(H^2(T,\CC))=n-1$. 
In this case, using Noether-Lefschetz theory, one can obtain  a series of
statements on the Mordell-Weil 
rank of a very general elliptic surface: e.g., one  obtains
statements on the Mordell-Weil rank for a very general degree $6n$ elliptic
surface,
and results on the dimension of the locus of elliptic surfaces with fixed
Mordell-Weil-rank \cite{CoxNL, KloNL}. 
However, if $n>1$ then $h^{2,0}(E)>0$ and hence it seems hard  to calculate 
$\rank
(H^{1,1}(H^2(E,\CC)) \cap H^2(E,\ZZ))-1$ in concrete examples. This is the key
obstruction for proving results similar to Corollary~\ref{corTrvMW}.

To calculate the rank of $\MW(\pi)$ we need to calculate the group $H^4(Y,\CC)$
together with its Hodge structure. If $Y$ has only isolated singularities and
all singularities are semi-weighted homogeneous  hypersurface singularities then
this
can be done by applying a method of Dimca \cite{DimBet}.
However, $Y$ might have non-isolated singularities. It turns out in our
situation that at a
general point of a one-dimensional component of $Y_{\sing}$ we have a
transversal  $ADE$ surface singularity. 
We extend Dimca's method to a class of hypersurfaces with non-isolated
singularities:

For the calculation of $H^4(Y,\CC)$ there is no reason to assume that the
hypersurface comes from an elliptic fibration,  i.e., at this stage we work in
the following context:
let $\PP=\PP(w_0,w_1,w_2,w_3,w_4)$ be a 4-dimensional weighted projective space
and set
$w=w_0+w_1+w_2+w_3+w_4$. We call a degree $d$ hypersurface $Y\subset \PP$
\emph{admissible} if $Y$ is defined 
by a weighted homogeneous
polynomial $f\in \CC[x_0,x_1,x_2,x_3,x_4]$, such that
\begin{enumerate}
\item $Y$ intersects $\PP_{\sing}$ transversally, i.e., if $\Sigma$ is the
locus where all the partials of $f$ vanish, then $\Sigma \cap
\PP_{\sing}=\emptyset$. ($Y$ will still have singularities along 
$\PP_{\sing}$, these arise from the construction of the weighted projective
space and
are finite quotient singularities.)
 \item $Y$ is smooth in codimension 1.
 \item In codimension 2 the threefold $Y$ has only transversal  $ADE$ surface
singularities.
 \item In codimension 3 all singularities are contact equivalent to a weighted 
 homogeneous hypersurface singularity (cf. Remark~\ref{remCont}).
\end{enumerate}

To formulate our theorem concerning the calculation of the cohomology groups we
have to introduce some notation:
we define $\cP$ as the set of all points $p\in \Sigma$, such that $(Y,p)$ is not
a transversal
 $ADE$ surface singularity.
Now let $f_p\in \CC[y_0,y_1,y_2,y_3]$ be such that $(f_p,0)$ is contact
equivalent to
$(Y,p)$, where $f_p$ is weighted homogeneous of degree $d_p$ and $w_p$ is the
sum
of
the weights.    In  particular, $f_p=0$ defines a surface in some weighted
projective 3-space. 

Let $R({f_p})$ be the Jacobian ring of $f_p$. If $(Y,p)$ is an isolated
singularity we set $\tilde{R}(f_p)=R(f_p)$. If $(Y,p)$ is not an isolated
singularity, then
$\tilde{R}(f_p)$ is defined as follows: the equation $f_p=0$ determines a
surface $S\subset\PP(v_0,v_1,v_2,v_3)$, which has finitely many
singularities $(S,q_1)$, $\dots$, $(S,q_t)$. Let $M_j$ be the Milnor-algebra of
$(S,q_j)$ and set $\mu:=\sum_j \dim M_j$ to be the total Milnor number. Let
$h_1,\dots,h_{\mu}$ be polynomials of degree $2d_p-w_p$, such that their image
under the natural (surjective) map $R(f_p)_{2d-w} \to \oplus_j M_j$ spans
$\oplus_j M_j$ and
set  $\tilde{R}(f_p)=R(f_p)/(h_1,\dots,h_\mu)$.

Using that $f_p=0$ is contact equivalent to $(Y,0)$ one obtains a natural map
$R(f)_{kd-w}\to
R(f_p)_{kd_p-w_p}$ for $k=1,2$.

The following theorem is a combination of Proposition~\ref{propLocCoh} and
several results from 
Section~\ref{secGlue}.
\begin{theorem}\label{singcohThm} Let $Y$ be an admissible hypersurface. Then
 \begin{eqnarray*} &H^1(Y,\QQ)=H^5(Y,\QQ)=0 \mbox{ and } \\& H^0(Y,\QQ)=\QQ,
H^2(Y,\QQ)=\QQ(-1), H^6(Y,\QQ)=\QQ(-3).\end{eqnarray*}  
The group
$H^4(Y,\QQ)$ has a pure weight $4$ Hodge structure, with vanishing
$h^{4,0}$ and $h^{0,4}$ and
\begin{eqnarray*}  h^{3,1}(H^4(Y,\CC))& =  &\dim \coker (R(f)_{d-w} \to
\oplus_{p\in \cP} \tilde{R}(f_p)_{d_p-w_p})\\
 h^{2,2}(H^4(Y,\CC)_{\prim})&=& \dim \coker (R(f)_{2d-w} \to \oplus_{p\in \cP}
\tilde{R}(f_p)_{2d_p-w_p}).
 \end{eqnarray*} 
\end{theorem}

Combining Theorems~\ref{thmCohRel} and~\ref{singcohThm} we obtain the following
(see also 
Section~\ref{secMethod})

\begin{theorem} Let $\pi: X \to S$ be an elliptic threefold, such that $S$ is a
rational surface, and the associated threefold $Y\subset \PP$ is admissible.
Assume that the map
\[ R(f)_{d-w} \to \oplus_{p\in \cP} \tilde{R}(f_p)_{d_p-w_p} \]
is surjective. Then
\[ \rank \MW(\pi) = \dim \coker (R(f)_{2d-w} \to \oplus_{p\in \cP}
\tilde{R}(f_p)_{2d_p-w_p}). \]
\end{theorem}

\begin{remark}
The only non-zero Betti number that we have not determined so far is  $h^3(Y)$.
Usually, one is able to calculate $e(Y)$ and one can use this to determine
$h^3(Y)$.
\end{remark}

 \begin{remark}
If $Y$ is not admissible then our method
fails. In this case the first step would be to calculate the local
cohomology $H^i_p(Y,\QQ)$ of such a singularity. To our knowledge there is no
method
which works for a large class of such singularities.
\end{remark}

This theorem can be used to classify elliptic threefolds with small numerical
invariants.
In \cite{jconst} we classify the possibilities for $\MW(\pi)$ if $n=1$ and the
$j$-invariant of the fibers of $\pi$ is constant.

Our method is similar to Dimca's, but differs from recent methods such as work
by Cynk \cite{Cynk}, Rams \cite{Rams}, Grooten-Steenbrink \cite{GrSte},  and
the classical work of Clemens \cite{CleDS}, Werner \cite{Wer}, Schoen
\cite{Sch} and van Geemen-Werner \cite{vGW}. 
 
The differences between the methods of the papers quoted above
and ours are the following:
in all cases the method is applied to a smaller class of singularities, namely
in
the isolated singularity case Rams  deals with isolated $A_k,D_m,E_n$
singularities. In the non-isolated case, Grooten-Steen\-brink deal with
transversal $A_1$ singularities and singularities of the type $w^2=xyz$ and
$zw=x^2y$. The other papers deal with a subset of these singularities.

The restriction on the type of singularity (by Rams and by Grooten-Steen\-brink) 
implies that $(R_{f_p})_{d-w}=0$ for
all singularities they consider. 
 In particular, $H^4(Y,\QQ)$ is a pure $(2,2)$ Hodge
structure.
A second difference between our method and the above mentioned methods is, that both Rams and 
Grooten-Steenbrink express $H^{2,2}(H^4(Y,\CC))$ as a cokernel of a map between two vector spaces, which are of 
larger dimension than the vector spaces that occur in the sequel.

The organization of this paper is as follows. In Section~\ref{secSetup} we
recall some standard facts on elliptic fibrations over rational varieties.  In
Section~\ref{secBir} we discuss 
some results of Miranda from \cite{MirEllThree} that allow us to describe the
rational map $X
\dashrightarrow Y$. In Section~\ref{secComMW}
we give  proofs of Theorem~\ref{thmCohRel} and Corollary~\ref{corTrvMW}.
In Section~\ref{secCoh} we recall some standard results on the cohomology of
hypersurfaces $Y$ in weighted projective space. In the case of non-quasismooth
hypersurfaces we use the Poincar\'e residue map to calculate the cohomology of
the smooth part of $Y$.  In Sections~~\ref{secSur},~\ref{secLoc}
and~\ref{secGlue} we relate the 
cohomology of the smooth part of $Y$ and some local cohomology with the
cohomology of $Y$. This 
enables us to prove Theorem~\ref{singcohThm}. 
In Section~\ref{secMethod} we
summarize our method to calculate the Mordell-Weil group.
The remaining sections are devoted to applications of our method.
In Section~\ref{secEasy} we calculate the Mordell-Weil rank in an example with non-isolated singularities.
 In Section~\ref{secBeh} we calculate the Mordell-Weil rank of a
class of elliptic Calabi-Yau threefolds which were constructed by Hirzebruch. 
This calculation allows us to compute
all the Hodge numbers of these threefolds.

\part{Relation between the Mordell-Weil group and cohomology of singular
hypersurfaces}
\section{Set-up}\label{secSetup}

\begin{definition} An \emph{elliptic threefold} is a quadruple
$(X,S,\pi,\sigma_0)$, with $X$ a smooth projective threefold, $S$ a smooth
projective surface, $\pi:X\to S$ a flat morphism, such that the generic fiber is
a genus 1 curve and $\sigma_0$ is a section of $\pi$.

The \emph{Mordell-Weil group} of $\pi$, denoted by $\MW(\pi)$, is the group of 
rational sections $\sigma: S \dashrightarrow  X$ with identity element
$\sigma_0$.
\end{definition}

Recall that  a morphism $\pi:X \to S$ (with $X$ a
smooth projective threefold and $S$ a smooth projective surface) is flat if and
only if all fibers have dimension one.
Clearly $\MW(\pi)$ is a birational invariant, in the sense that if $\pi_i:
X_i\to S_i$, $i=1,2$ are elliptic threefolds with zero-sections $\sigma_0$ and $\sigma_0'$ such that there exist an birational
isomorphism $\psi: X_1\stackrel{\sim}{\dashrightarrow} X_2$ mapping the general
fiber of $\pi_1$ to the general fiber of $\pi_2$ and such that $\psi\circ \sigma_0=\sigma_0'$ then $\psi^*: \MW(\pi_2) \to
\MW(\pi_1)$ is well-defined and is an isomorphism. Moreover, the rank of $\MW(\pi)$ is also stable under base-change by a birational morphism on the base surface.

The following technical definition will be needed
\begin{definition} Let $\pi: X \to S$ be an elliptic threefold. An effective
divisor $D\subset X$ is called \emph{fibral} if $\pi(D)\subset S$ is a curve.
\end{definition}

We shall frequently make use of the following fundamental result:
\begin{theorem}[{Shioda-Tate-Wazir, \cite[Theorem 4.2]{Waz}}]\label{thmSTW}
Let $\pi :X \to S$ be an elliptic threefold then
\[ \rho(X)=\rho(S)+f+\rank \MW(\pi)+1\]
where $f$ is  the number
of 
irreducible surfaces $F$ in $X$ such that
$\pi(F)$ is a curve, and $F\cap \sigma_0(S)=\emptyset$.
\end{theorem}

Using Lefschetz' (1,1) theorem and Poincar\'e duality we can rephrase the
Shioda-Tate-Wazir formula as 
\[ \rank \MW(\pi) = \rank H^{2,2}(X,\CC) \cap H^4(X,\ZZ)-f-\rho(S)-1.\]
In general this is hard to compute. Theorem~\ref{thmCohRel} says that the
analogous formula also holds if we replace $X$ by a minimal (singular)
Weierstrass model. In this case one has tools to compute the right hand side.

We shall now describe in some detail how to associate to
an elliptic threefold $\pi: X \to S$ a hypersurface in weighted projective
$4$-space.
Here we restrict
ourselves to the case where $S$ is a rational surface.  In this case we can find
a hypersurface $Y$ of degree $6n$ in $\PP(2n,3n,1,1,1)$ which is birational to
$X$ as follows:
the morphism $\pi$ establishes $\CC(X)$ as a field extension of
$\CC(S)=\CC(z_1,z_2)$.
The field $\CC(X)$ is the function field of an elliptic curve over
$\CC(z_1,z_2)$,
i.e., $\CC(X)=\CC(x,y,z_1,z_2)$ where \begin{equation}\label{eqnWei}
y^2=x^3+f_1(z_1,z_2)x+f_2(z_1,z_2)\end{equation} with $f_1,f_2\in
\CC(z_1,z_2)$.
Without loss of generality we may assume that (\ref{eqnWei}) is a global minimal
Weierstrass equation, i.e.,  $f_1,f_2$ are polynomials and there is no
polynomial $g\in \CC[z_1,z_2]$ such that $g^4$ divides $f_1$ and $g^6$ divides
$f_2$. 
 
To obtain a hypersurface in $\PP(2n,3n,1,1,1)$ we need to find a weighted
homogeneous polynomial. Let $n=\lceil \max  \{\deg(f_1)/4,\deg(f_2)/6\} \rceil$ 
and define $P$ and $Q$ as the polynomials 
 \[P=z_0^{4n} f_1(z_1/z_0,z_2/z_0), \quad Q=z_0^{6n} f_2(z_1/z_0,z_2/z_0).\]
  Then
 \[ y^2=x^3+P(z_0,z_1,z_2)x+Q(z_0,z_1,z_2)\]
 defines a hypersurface $Y$ of degree $6n$ in $\PP:=\PP(2n,3n,1,1,1)$. Let
$\Sigma$ be the locus 
where all the partial derivatives of the defining equation vanish.
Consider the projection $\tilde{\psi}: \PP(2n,3n,1,1,1) \dashrightarrow \PP^2$
with center
$L=\{z_0=z_1=z_2=0\}$ and its restriction $\psi=\tilde{\psi}|_Y$ to $Y$. 
Then there exists a diagram
\[
 \xymatrix{X \ar@{-->}[r] \ar[d]^\pi & Y\ar@{-->}[d]^\psi \\ S \ar@{-->}[r]&
\PP^ 2.  }
 \] 

Note that $Y\cap L=\{(1:1:0:0:0)\}$. 
If $n=1$ then $\PP_{\sing}$ consists of two points, none of which lie on
$Y$. If $n>1$ then  an easy calculation in local coordinates shows that
$\PP_{\sing}$
is precisely $L$, that $\Sigma$ and $L$ are disjoint  and that $Y$ has an
isolated singularity at 
$(1:1:0:0:0)$. 
For any $n$ we have that $\psi$ is not defined at $(1:1:0:0:0)$. 
Let $\tilde{\PP}$ be the blow-up of $\PP$ along $L$.
Let $X_0$ be the strict transform of $Y$ in $\tilde{\PP}$. An easy calculation
in local coordinates shows that $X_0\to Y$ resolves the singularity of $Y$ at
$(1:1:0:0:0)$ and that the induced map $\pi_0: X_0 \to S_0$ with
$S_0= \PP^2$ is a morphism. Moreover, all fibers of $\pi_0$ are irreducible
curves.

\section{Miranda's construction}\label{secBir}

The threefolds $X_0$ and $X$ are birational and one might therefore ask for a 
precise sequence of
birational morphisms relating $X_0$ and $X$. This question might be too hard. A
slightly weaker problem is solved by Miranda: starting with $\pi_0:X_0\to S_0$
Miranda \cite{MirEllThree} produces a smooth elliptic threefold $\pi': X' \to
S'$
birational to $\pi$.
Actually,  Miranda produces a series $\{\pi_i:X_i\to S_i \}$ where
$\{\pi_{i+1}:X_{i+1} \to S_{i+1} \}$ can be obtained from $\{\pi_i:X_i\to S_i
\}$ by
applying one of the following three types of birational transformations:

\begin{enumerate}
\item $S_{i+1}$ is the blow-up of $S_{i}$ in a point $p$ of the discriminant
curve of $\pi$, i.e., 
with $\pi_i^{-1}(p)$ a singular curve. Then  we define $X_{i+1}$ as the fiber
product of 
$X_{i}$ with $S_{i+1}$ over
$S_{i}$:
\[ \xymatrix{X_{i+1}:=X_{i} \times_{S_{i}} S_{i+1} \ar[d] \ar[r] & X_{i}\ar[d] 
\\S_{i+1}:=\Bl_p S_{i}  \ar[r] & S_{i}.\ }\]

This procedure is applied in the following two cases
\begin{enumerate}
\item To simplify the geometry:  let $\Delta_{i}\subset S_{i}$ be the (reduced)
discriminant curve of $\pi_i$. After applying this procedure sufficiently many
times, we may assume that each irreducible component of $\Delta_i$ is smooth,
and that $\Delta_i$ has only ordinary double points as singularities.
\item Suppose $X_{i}$ has  an isolated singularity in the fiber of $p\in S_{i}$.
Blowing up this singularity would yield a non-flat morphism. Instead, if we
apply this base change procedure we get a curve of singular points in
$X_{i+1}$. 
\end{enumerate}

\item Even when we start with a minimal local equation, we might obtain a
non-minimal
equation, i.e.,  it might happen that  $X_{i}$ has, in one of its charts, a
local
equation of the form by 
$y^2=x^3+u^4f_1x+u^6 f_2$,
where $f_1,f_2\in \CC[z_0,z_1]$ and $u\in \CC[z_0,z_1]\setminus \CC$ is
irreducible. 
In this chart the elliptic fibration is given by $(x,y,z_0,z_1) \mapsto
(z_0,z_1)$, which can be interpreted as projection 
onto the plane $x=y=0$.
Note that  after
applying the first operation sufficiently many times, we can assume that
$x=y=u=0$
is a smooth irreducible curve. 
We need to get rid of the factor $u^4$ and $u^6$ in the
equation, which can be done as follows:
\begin{enumerate}
\item Blow up $C_i: x=y=u=0$, yielding a threefold $X_{i+1}$ with local
equation 
$y^2=ux^3+u^3f_1 x+u^4 f_2$ in one of the charts. An easy calculation shows that
in the other two 
``new'' charts we have that $X_{i+1}$ is smooth.
\item Blow up $C_{i+1}: x=y=u=0$, yielding a (non-normal) threefold $X_{i+2}$
with local
equation $y^2=u^2x^3+u^2f_1 x+u^2 f_2$ in one of the charts.
\item Blow up the surface $R_{i+2}: u=y=0$, yielding a threefold $X_{i+3}$ with
local
equation $y^2=x^3+f_1x+f_2$ in one of the charts.
\item If we patch all the local charts together, we see that the fiber over a
point in $\{u=0\}$ is a reducible curve, consisting of two rational curves and
one
elliptic curve. Actually $\pi^{-1}_{i+3}(\{u=0\})$ consists of three irreducible
components, two of them are ruled surfaces over $C: \{u=0\}$, the third is an
elliptic
surface. We can contract the two ruled surfaces, obtaining $X_{i+5}$.

An easy calculation in local coordinates shows that both $X_{i+3}\to X_{i+4}$
and $X_{i+4}\to X_{i+5}$ are blow-ups with center a smooth curve contained in
the smooth 
locus.
\end{enumerate}
The base surface remains unchanged, i.e., $S_i=S_{i+1}=\dots =S_{i+5}$. The
geometric construction is summarized in the following table: 
\[
\begin{array}{ccc}
\mbox{Threefold} & \mbox{Singular locus} & \mbox{Important divisor }\\
X_i           & C_i \mbox{ (curve)} &   F_i =\pi_{i}^{-1}(\{u=0\}) \\
X_{i+1}=\Bl_{C_{i}} (X_i) & C_{i+1} \mbox{ (curve)} &
E_{i+1}/C=\PP^1-\mbox{bdle}.\\
X_{i+2}=\Bl_{C_{i+1}} (X_{i+1}) & R_{i+2}=E_{i+2} \mbox{ (surface)} &  
E_{i+2} /C=\PP^1-\mbox{bdle}.\\
X_{i+3}=\Bl_{R_{i+2}} (X_{i+2}) & \emptyset &  E_{i+3}=\mbox{elliptic surface}\\
&& \mbox{double cover of } E_{i+2} \\
X_{i+4} = \CCon_{E_{i+1}}(X_{i+3}) &&\\
X_{i+5} = \CCon_{F_i}(X_{i+4}) &&\\
\end{array}
\]
When we contract $E_{i+1},F_i$ we mean that we contract the strict transform of
$E_{i+1},F_i$.

\item To resolve  singularities: $X_{i+1}$ is obtained by blowing up a  curve
$C$
inside the singular locus of $X_{i}$ such that $C_{\red}$ is smooth. Set
$S_{i+1}=S_{i}$ and $\pi_{i+1}$ to be the composition $X_{i+1}\to
X_{i}\stackrel{\pi_{i}}{\to} S_{i}$.

Note that by using the defining equation one can show that at a general point of
$C_{\red}$
one has a transversal $ADE$ surface singularity.
 \end{enumerate}

These three steps should be applied in the following order:

\begin{enumerate}
\item Apply step 1, to obtain a fibration with nice properties: i.e., 
repeat step 1 until
$\Delta_{i,\red}\subset S_i$ has at most nodes as singularities and  
the $j$-function $j: 
S_i\dashrightarrow \PP^1$  is a morphism.

At this stage we obtain a Weierstrass{} fibration i.e., there exists a line
bundle $\cL_i$ on $S_i$ and 
sections $A \in H^0(S_i,\cL_i^{\otimes 4}), B\in H^0(S_i,\cL_i^{\otimes 6})$
such that $X_i = \{ Y^2Z=X^3
+AXZ^2+BZ^3\} \subset \PP(\cO \oplus \cL_i^{-2} \oplus \cL_i^{-3})$. We can
consider $A=0$ and $B=0$ as curves inside $S_i$. Repeat step 1 until the reduced
curves underlying $A=0$ and $B=0$ have at most ordinary double points as
singularities.

\item Apply step 2, until there is no curve $C\subset S_i$ such that $A$
vanishes along $C$ with order at 
least 4, and $B$ vanishes along $C$ with order at least 6.

\item \label{repeat} Apply step 3, until $X_i$ has only isolated singularities
or is smooth. If $X_i$ is 
smooth then stop. 

\item Apply step 1 for each of the isolated singularities of $X_i$.  
The outcome of this is a threefold whose singular locus consist of finitely many
smooth irreducible 
curves which are all disjoint.

\item If necessary  apply step 2.
\item Go to point (\ref{repeat}).
\end{enumerate}

{}From this description it is not at all clear  why this procedure should
terminate. For this fact we refer to 
\cite{MirEllThree}.

\begin{remark} Miranda uses a slightly different order and he uses a fourth type
of modification, namely 
the
contraction of $\PP^1\times \PP^1$ to a $\PP^1$. We indicate now why this does 
not influence the termination of this procedure.
  
  The extra modification is applied if
$X_{i}$ has an isolated 
$A_1$ singularity at $p\in X_i$. We can then first blow up $X_i$ in $p$. The
exceptional divisor $E$ is 
isomorphic  to $\PP^1\times \PP^1$. The morphism $\pi_{i+1}: X_{i+1}\to
S_{i+1}=S_i$ has a fiber with a 
two-dimensional component, contradicting flatness. This can be resolved by
contracting $E$ to $\PP^1$, 
a so-called ``small resolution". 
The problem is that the space  $X_{i+2}$ obtained in this way is a priori only
an algebraic space, rather 
than an 
algebraic variety. To determine whether $X_{i+2}$ is actually an algebraic
variety one needs to consider 
the global geometry of $X_{i+2}$.

To avoid this problem  we choose a different procedure: namely we blow up
$S_{i}$ in $\pi_i(p)$ and 
then base change. The threefold $X_{i+1}$ now has a curve  $C$ of singularities.
Then we blow up $C$ 
and obtain a threefold $X_{i+2}$. A direct calculation in local coordinates 
shows that $X_{i+2}$ is 
smooth in a neighborhood of the exceptional divisor of $X_{i+2}\to X_{i+1}$. We
give a sketch of this 
calculation:
in local coordinates $(X_i,p)$  is given by
$t_1^2+t_2^2+t_3^2+t_4^2=0$. 
If we use the base change procedure, we obtain a curve $C\subset X_{i+1}$ of
singularities. A 
straightforward 
calculation shows that at a general point of $C$ we have a local equation of the
form
$s_1^2+s_2^2+s_3^2=0$, i.e., we have a transversal $A_1$ surface singularity,
except for two points on 
$C$ 
where we have a local equation of the form $s_1^2+s_2^2+s_4s_3^2=0$ (a so-called
pinch point). 
Here $C$ is given by the equation $s_1=s_2=s_3=0$.

Following the above algorithm, we now need to blow up $C$. A calculation in
local coordinates shows 
that the threefold $X_{i+2}$ obtained in this way is smooth in a neighborhood of
the exceptional divisor.

In order to show that our procedure terminates, note that one could follow
Miranda's algorithm until one has only isolated $A_1$-singularities left.
It is clear that the above procedure then resolves all the remaining
singularities.
\end{remark}

\section{Comparing Mordell-Weil ranks}\label{secComMW} 

Starting with an elliptic threefold $\pi:X \to S$ we found a hypersurface
$Y\subset \PP(2n,3n,1,1,1)$. Applying Miranda's construction to $Y$ gives us an 
elliptic threefold $\pi': X'\to S'$.  We now want to express 
$\rank \MW(\pi)=\rank \MW(\pi')$ in terms of invariants of $Y$. For this we use
the
following result:
\begin{theorem} \label{thmMV} Let $V$ and $\tilde{V}$ be complex varieties. 
Let $\varphi: \tilde{V} \to V$ be a proper
birational morphism. Let $\cZ\subset V$ be a closed subvariety such that
$\varphi$ restricted to $\tilde{V}\setminus \pi^{-1}(\cZ)$ is injective. Set
$E:=\pi^{-1}(\cZ)$. Then there is an exact sequence of Mixed Hodge structures
\[ \dots \to H^{i-1}(E,\QQ) \to H^i(V,\QQ) \to H^i(\tilde{V},\QQ)\oplus
H^i(\cZ,\QQ)\to H^i(E,\QQ)  \to \dots.\]
 \end{theorem}
\begin{proof}
See \cite[Corollary 5.37]{PS}.
\end{proof}

\begin{lemma} \label{lemBlowUp}
Let $V$ be a threefold, $C\subset V$  be a smooth curve contained in the smooth
locus of $V$. 
Let $\tilde{V}$ be the blow-up of $V$ along $C$, let $E$ be the exceptional
divisor and $\iota:E\to \tilde{V}$ be the inclusion.
Then
\[ \iota^*: H^3(\tilde{V},\QQ)\to H^3(E,\QQ)\]
is surjective.
\end{lemma}
\begin{proof}
Let $\psi:V_1\to V$ be a resolution of singularities of $V$ and let $E_1$ be the
exceptional divisor of $\psi$. 
Since $C$ is contained in the smooth locus we have that $\psi^{-1}(C)$ is
isomorphic to $C$. Let $\psi_1: \tilde{V_1}\to V_1$ be the blow-up of $V_1$
along $\psi^{-1}(C)$. Equivalently, $\tilde{V_1}=\tilde{V} \times_V V_1$.

The exceptional divisor of $\psi_1$ is isomorphic to $E$ and the exceptional
divisor of $\tilde{V_1}\to V$ is isomorphic to the disjoint union of $E$ and
$E_1$. Denote $\Sigma=V_{\sing}$.

{}From Theorem~\ref{thmMV} we get the following exact sequence 
\[ \dots \to H^3(\tilde{V_1},\QQ) \to H^3(E,\QQ) \to H^4(V_1,\QQ) \to \dots.\]
Since $V_1$ and $E$ are smooth we have that $H^3(E,\QQ)$ has a pure weight 3
Hodge structure and $H^4(V_1,\QQ)$ has a pure weight 4 Hodge structure. Hence
the map $H^3(E,\QQ)\to H^4(V_1)$ is the zero map and $H^3(\tilde{V_1},\QQ)\to
H^3(E,\QQ)$ is surjective.
Consider now the exact sequence of Theorem~\ref{thmMV} for $\psi_1\circ \psi$: 
\[\dots\to  H^3(\tilde{V_1},\QQ) \oplus H^3(\Sigma,\QQ) \to H^3(E_1,\QQ) \oplus
H^3(E,\QQ) \to H^4(V,\QQ) \to \dots\]
Since $H^3(\tilde{V_1},\QQ)\to H^3(E,\QQ)$ is surjective we obtain that
$H^3(E,\QQ)\to H^4(V,\QQ)$ is the zero map.

Consider now the exact sequence of Theorem~\ref{thmMV} for $\tilde{V} \to V$:
\[\dots\to  H^3(\tilde{V},\QQ) \to H^3(E,\QQ) \to H^4(V,\QQ) \to \dots\]
Since $H^3(E,\QQ) \to H^4(V,\QQ)$ is the zero map we obtain that
$H^3(\tilde{V},\QQ) \to H^3(E,\QQ)$ is surjective.
\end{proof}

\begin{theorem}\label{thmMW} Let $Y\subset \PP$ be a minimal Weierstrass
fibration and let $\pi: X \to S$ be an elliptic threefold, birational to $Y$.
Assume that $H^4(Y,\QQ)$ has a pure weight $4$ Hodge structure. Then	
\[\rank \MW(\pi)= \rank \left(H^{2,2}(H^4(Y,\CC))\cap  H^4(Y,\ZZ)\right)  -1\]
and
$H^5(Y,\QQ)\cong H^5(X,\QQ)$.
\end{theorem}
\begin{proof} Since both $\rank \MW(\pi)$ and $H^5(X,\QQ)$
are birational invariants of smooth fibred threefolds, it
suffices to prove this  statement for the elliptic threefold $\pi': X' \to S'$
obtained from
Miranda's procedure.
  Then by the
Shioda-Tate-Wazir formula and Lefschetz (1,1) one has
\begin{eqnarray*} 
\rank \MW(\pi) &=& \rho(X')-\rho(S')-f-1 \\
& =& \rank H^2(X',\ZZ) \cap H^{1,1}(X',\CC)  -\rho(S')-f-1 \\
& =& \rank H^4(X',\ZZ) \cap H^{2,2}(X',\CC)  -\rho(S')-f-1
\end{eqnarray*}
where $f$ is the number of independent fibral divisors, not intersecting the
image of the zero
section. 

Let $\pi_i:X_i\to S_i$ be the associated sequence of modifications. Let $f_i$
denote the number of 
independent fibral  divisors of $\pi_i$, not
intersecting the zero-section. 
We will show by induction that for each $i$ we have that $H^4(X_i,\QQ)$ has a
pure weight 4 Hodge structure and that
\begin{equation}\label{eqninv}
 \rank \left(H^{2,2}(H^4(X_i,\CC))\cap  H^4(X_i,\ZZ)\right) - \rho(S_i)-f_i-1
\end{equation}
is independent of $i$. 

This suffices for the first statement: 
for the elliptic threefold in the final step of Miranda's construction 
we have that (\ref{eqninv}) equals $\rank \MW(\pi)$
by the Shioda-Tate-Wazir formula. 

Now consider (\ref{eqninv}) for $i=0$. {}From
$S_0=\PP^2$ we get  $\rho_0(S_0)=1$. 
Since all fibers of $\pi_0$ are irreducible, we get $f_0=0$. Finally,
Theorem~\ref{thmMV} applied to 
$X_0\to Y$ yields an exact sequence of $\QQ$-MHS
\[H^3(E,\QQ)\to  H^4(Y,\QQ)\to H^4(X_0,\QQ) \to H^4(E,\QQ) \to H^5(X_0,\QQ). \]
Since $E\cong \PP^2$ we get $H^3(E,\QQ)=0$ and $H^4(E,\QQ)=\QQ(-2)$. Also the
map $H^4(X_0,\QQ)\to H^4(E,\QQ)$ is non-zero, hence we get
\[ 0 \to H^4(Y,\QQ) \to H^4(X_0,\QQ) \to \QQ(-2) \to 0.\]
In particular, $H^4(X_0,\QQ)$ has a pure weight 4 Hodge structure and
\begin{eqnarray*} 
&& \rank \left(H^{2,2}(H^4(X_0,\CC))\cap  H^4(X_0,\ZZ)\right) - \rho(S_0)-f_0-1 
\\ 
&=& \rank \left(H^{2,2}(H^4(X_0,\CC))\cap  H^4(X_0,\ZZ)\right) - 2  \\
&=& \rank \left(H^{2,2}(H^4(Y,\CC))\cap  H^4(Y,\ZZ)\right) - 1. 
\end{eqnarray*}

To prove that (\ref{eqninv}) is actually independent of $i$, we consider each of
the
three types of modifications mentioned in Miranda's construction separately. 
In each case we
apply Theorem~\ref{thmMV} several times without mentioning it explicitly:
\begin{enumerate}
\item Consider the first type of modification, i.e. we blow up a point 
$p\in \Delta \subset S_i$ and then base change. For 
the proper modification $X_{i+1}\to X_i$ we have that $\cZ=C\subset X_i$
 is a curve of
arithmetic genus 1, i.e., $C$ is either a union of $k$ rational curves, a 
cuspidal rational curve or a nodal rational curve. 
In the last two cases we set
$k=1$.  Using the universal property of the fiber product we obtain that 
the exceptional divisor $E\subset X_{i+1}$
is isomorphic to a product $C\times \PP^1$.  
Using our induction hypothesis on
$H^4(X_i,\QQ)$ (i.e., that it is of pure weight 4) and that $H^3(E,\QQ)$ has no
classes
of weight $\geq 4$  
\cite[Theorem 5.39]{PS}, the exact sequence of Theorem~\ref{thmMV} yields the
following exact sequence
 \[0 \to  H^4(X_i,\QQ)\to  H^4(X_{i+1},\QQ) \to  H^4(E,\QQ)=\QQ(-2)^k.\]
Each of the $k$ 
irreducible components of $C\times \PP^1$  yields a class $\xi_j$ in
$H^4(X_{i+1},\QQ)$. I.e., we have
  \[  \spa \{\xi_1,\dots,\xi_k\} \subset H^4(X_{i+1},\QQ) \to H^4(E,\QQ) \]
Clearly $\dim H^4(E,\QQ)=k$ and the $\xi_j$ map to a basis
of $H^4(E,\QQ)$. In particular, the $\xi_j$ are
independent in $H^4(X_{i+1},\QQ)$ and the map $H^4(X_{i+1},\QQ) \to H^4(E,\QQ)$
is surjective.  The
conclusion is that 
 \[\rank \left(H^{2,2}(H^4(X_{i+1},\CC))\cap  H^4(X_{i+1},\ZZ)\right) = \]
\[ = k+ \rank
\left(H^{2,2}(H^4(X_{i},\CC))\cap  H^4(X_{i},\ZZ)\right),  \]
  $f_{i+1}=f_{i}+k-1$ and  $\rho(S_{i+1})=\rho(S_i)+1$, and hence the quantity
(\ref{eqninv}) is unchanged.
 
\item  
The second modification consists of two blow-ups of a curve, the blow-up
of a rational surface and two blow-down morphisms. 
We consider first the blow-up of a curve in $X_i$, and the blow-up of the curve
in $X_{i+1}$.
A reasoning very similar to the previous case yields  that
$H^4(X_{i+1},\QQ)$ and $H^4(X_{i+2},\QQ)$ have a pure weight 4 Hodge structure,
that classes of type  $(2,2)$ are added to $H^4(X_{i+1},\ZZ)$ and
$H^4(X_{i+2},\ZZ)$
and that $f_{i+2}=f_{i+1}+1=f_i+2$.
I.e., the quantity (\ref{eqninv}) is unchanged.

Consider now the third  step, the blow-up of a rational surface.
In this case both $\cZ$ and $E$ are irreducible surfaces and we have an
isomorphism  $H^4(\cZ,\QQ) \to H^4(E,\QQ)$. Since $H^3(E,\QQ)$ has Hodge weights
at most
$3$ \cite[Theorem 5.39]{PS} and  $H^4(X_{i+2},\QQ)$ has a pure weight 4 Hodge
structure,
 Theorem~\ref{thmMV} implies that we have  an
isomorphisms $ H^4(X_{i+2},\QQ) \to  H^4(X_{i+3},\QQ)$. 
Hence $H^4(X_{i+1},\QQ)$ is of pure weight 4 and all
entries in (\ref{eqninv}) remain unchanged.

The final two steps are the contraction of the two ruled surfaces. I.e.,
$X_{i+3} \to X_{i+4}$ and $X_{i+4} \to X_{i+5}$ are blow-ups of curves. In the
previous section it is argued that these curves are smooth and lie in the smooth
locus of $X_{i+4}$ and $X_{i+5}$.

 Combining Lemma~\ref{lemBlowUp} with the exact sequence of Theorem~\ref{thmMV}
yields  exact sequences
\[ 0 \to H^4(X_{i+4},\QQ) \to H^4(X_{i+3},\QQ) \to H^4(E_{i+1},\QQ) \to \dots \]
and
\[ 0 \to H^4(X_{i+5},\QQ) \to H^4(X_{i+4},\QQ) \to H^4(F_i,\QQ) \to \dots \]
(notation as in the previous section.)

In particular, $H^4(X_{i+4},\QQ)$ and $H^4(X_{i+5},\QQ)$ have pure weight 4
Hodge structures. As above, one can show that the class of $E_{i+1}$ (resp.
$F_i$) in $H^4(X_{i+3},\QQ)$ (resp.  $H^4(X_{i+4},\QQ)$) is mapped to a nonzero
element in $H^4(E_{i+1},\QQ)$ ( resp. $H^4(F_i,\QQ)$). Hence these maps are
surjective, i.e., $H^4(X_{i+5},\ZZ)$ has rank 1 smaller than $H^4(X_{i+4},\ZZ)$,
and the difference is a class of type $(2,2)$.
Similarly, $H^4(X_{i+4},\ZZ)$ has rank 1 smaller than $H^4(X_{i+3},\ZZ)$, and
the difference is a class of type $(2,2)$.
 Moreover, $f_{i+3}=f_{i+4}+1=f_{i+5}+2$, hence the quantity (\ref{eqninv}) is
unchanged.

\item 
The third modification is to blow up a curve $C$ inside $X_{i ,\sing}$
such that $C_{\red}$ is smooth.
The exceptional divisor of such a blow up is not necessarily irreducible, say it
has $k$ irreducible components, hence $H^4(E,\QQ)=\QQ(-2)^k$. Each component of
$E$
yields a class $\xi_j$ in $H^4(X_{i+1},\QQ)$ and the same argument as above
shows that $H^4(X_{i+1},\QQ)$ has pure weight 4 and 
that the classes $\xi_j$ are independent. Hence $f_{i+1}=f_i+k$ and 
$\rank \left(H^{2,2}(H^4(X_i,\CC))\cap  H^4(X_i,\ZZ)\right)$
increases by $k$. Since $S_{i+1}=S_i$ we have proved that (\ref{eqninv}) remains
unchanged.
\end{enumerate}

To prove that $H^5(Y,\QQ)\cong H^5(X,\QQ)$, note that in all three cases the map
$H^4(X_i,\QQ)\to H^4(E,\QQ)$ is surjective. Since $h^5(\cZ,\QQ)=h^5(E,\QQ)=0$ it
follows from Theorem~\ref{thmMV} that $H^5(X_i,\QQ)\cong H^5(X_{i+1},\QQ)$ for
all
$i$.
\end{proof}

 \begin{corollary}\label{corTrvMWTxt} Let $\pi: X \to S$ be an elliptic
threefold associated with a hypersurface
 \[ y^2=x^3+Px+Q  \]
  with $P\in \CC[z_0,z_1,z_2]_{4n}$ and $Q\in \CC[z_0,z_1,z_2]_{6n}$, such that 
 \begin{enumerate}
\item the curve $\Delta:4P^3+27Q^2=0$ is reduced, $\Delta$ has only
double points as singularities, and $Q$ vanishes at each of these double points
or 
\item $P$ is identical zero  and  $Q=0$ defines a smooth curve of degree $6n$
in $\PP^2$.
 \end{enumerate} 
Then $\rank \MW(\pi)=0$.
 \end{corollary}

\begin{proof} Using Lefschetz hyperplane Theorem \cite[Theorem B22]{Dim} we
obtain that $h^2(Y)=1$. 
An easy calculation shows that our assumptions on $P$ and $Q$ are equivalent to 
 $Y$ being quasismooth. Then \cite[Corollary B19]{Dim} states that $H^i(Y,\QQ)$
satisfies Poincar\'e duality, hence
\[ h^4(Y)=h^2(Y)=1\]
and
$\rank\MW(\pi)=0.$
\end{proof}

\part{Cohomology of hypersurfaces in $\PP$}
\section{Cohomology of hypersurfaces in $\PP$: general results}\label{secCoh}

In this section let $Y$ be an irreducible and reduced hypersurface of degree $d$
in some weighted projective space $\PP$ of dimension $n+1$ defined by the
polynomial $g$. 
Let $\Sigma\subset \PP$ denote the locus where all the partials of $g$
vanish. We assume that $\Sigma$ does not
intersect $\PP_{\sing}$, i.e., $Y$ intersects the singular locus of $\PP$
transversally.
As usual we set $\dim \emptyset=-1$.

For an arbitrary hypersurface $Y$ the following form of Lefschetz'
hyperplane theorem holds:
\begin{proposition}[{\cite[Corollary B22]{Dim}}]\label{prpLHT} We have the
following isomorphisms for the cohomology of $Y$:
\begin{enumerate}
\item $H^i(Y,\QQ) \cong H^i(\PP,\QQ)$ for $i\leq n-1$.
\item $H^i(Y,\QQ) \cong H^i(\PP,\QQ)$ for $n+2+\dim \Sigma \leq i \leq 2n$.
\end{enumerate}
\end{proposition}

In all our applications we have $\dim \Sigma \leq 1$.

Let $U:=\PP \setminus Y$.
Since $U$ is affine 
we have 
\[ H^k(U,\CC)=H^0(U,\Omega^k_U)/d H^0(U,\Omega^{k-1}_U).\] 
Note that 
\[ H^0(U,\Omega^k_U) \cong \cup_{i\geq 0} H^0(\PP,\Omega^k_\PP(iY)).\]
For $\omega\in H^0(U,\Omega_U^k)$ define $\ord_Y(\omega) :=\min \{ i \colon
\omega
\in H^0(\PP,\Omega^k_{\PP}(iY) \}$. Let $P^\bullet$ be defined by
\[ P^s H^0(U,\Omega_U^k) = \{ \omega \in H^0(U,\Omega_U^k) \colon
\ord_Y(\omega)\leq k-s+1\}. \] 
Since $d(P^s H^0(U,\Omega^{k-1}_U)) \subset P^s H^0(U,\Omega^{k}_U)$ 
this induces a filtration $P^\bullet$ on $H^k(U,\CC)$, called the polar
filtration.

>From a result of Griffiths-Steenbrink (\cite[Section 4]{SteQua}) it follows that the Hodge filtration and the Polar filtration coincide if $Y$ is quasismooth. 
If we drop the assumption that $Y$ is quasismooth then we get the following
weaker 

\begin{theorem}[{Deligne-Dimca \cite{DelDim}}]\label{DD} For any hypersurface
$Y\subset \PP$ we have
\[P^sH^k(U,\CC) \supset F^sH^k(U,\CC).\]
\end{theorem}
There exist examples for which both filtrations differ, see \cite[Remark
6.1.33]{Dim}, \cite{DimGri}.

\begin{remark}\label{rmkminpole}
Since $H^{n+1}(U,\CC)=F^1 H^{n+1}(U,\CC)$ it follows from the above theorem that
$H^{n+1}(U,\CC)= P^{1}H^{n+1}(U,\CC)$. This implies that every 
 class of $H^{n+1}(U,\CC)$ has pole order at most $n+1$.
\end{remark}

The de Rham complex with filtration $P^\bullet$ yields
a spectral sequence $E_r^{p,q}$. In the quasismooth case 
this spectral sequence degenerates at $E_1$ and establishes natural isomorphisms between graded pieces of the Hodge filtration and certain graded pieces of the Jacobian Ring of $g$.

In the sequel we need the following notation. Let $x_i$ denote the
coordinates on
$\PP$ of weight $w_i$ and let $w=\sum w_i$. Set
\[ \Omega:=\left( \prod_j{x_j}\right)\sum (-1)^i w_i \frac{dx_0}{x_0}\wedge
\frac{dx_1}{x_1}\wedge \dots \wedge \widehat{\frac{dx_i}{x_i}} \wedge \dots
\wedge \frac{dx_{n+1}}{x_{n+1}}. \]
Then $H^0(\PP,\Omega^{n+1}(kY))$ is generated (as $\CC$-vector space) by
\[ \omega_f:=\frac{f}{g^k} \Omega\]
where $\deg(f)=kd-w$.

Write $Y^*=Y\setminus \Sigma$ and let $\PP^*=\PP\setminus \Sigma$ where, as
before, $\Sigma$ is 
defined
by the vanishing of the partials of $g$. Note that, since we have assumed that
$Y$ intersects $\PP_
{\sing}$
transversally, we have $\Sigma \cap \PP_{\sing}= \emptyset$.  In particular, 
$U=\PP^* \setminus Y^*=\PP\setminus Y$.

In generalizing the approach of Griffiths and Steenbrink to the non-quasismooth case we encounter the following
problems:
\begin{enumerate}
\item The Poincar\'e residue map is not an isomorphism.
\item We can still define the filtered de Rham complex and construct the
spectral sequence $E^{p,q}_r$. This sequence, however, does not degenerate at
$E_1$ but at
a higher step. 
\item The polar filtration and the Hodge filtration can differ.
\end{enumerate}

The following approach is similar to \cite{DimBet}, where Dimca studied
hypersurfaces with isolated singularities. The exact sequence of the pair
$(Y,Y^*)$ reads as 
\begin{equation} \label{eqnPairSigma} \dots \to H^{k}_\Sigma(Y,\QQ) \to
H^{k}(Y,\QQ) \to H^{k}(Y^*,\QQ) \to H^{k+1}_\Sigma(Y,\QQ) \to \dots
\end{equation}
This is a sequence of Mixed Hodge structures by \cite[Proposition 5.47]{PS}.

{}From now we on assume that $n=2$ and $\dim \Sigma\leq 0$ or $n=3$ and $\dim
\Sigma
\leq 1$. This will be the case in all our applications.
By Proposition~\ref{prpLHT} the only interesting cohomology groups are
$H^i(Y,\QQ)$ for $i=n,n+1,n+2$. 
We will study these groups by using (\ref{eqnPairSigma}). In this section we
focus on the calculation of $H^i(Y^*,\QQ)$.
The calculation of $H^i_{\Sigma}(Y,\QQ)$ will then be done in the following
sections.

We start by relating the cohomology of $Y^*$ to the cohomology of $U$ and
$\Sigma$. For this we need the notion of primitive cohomology. 
If $V \subset \PP$ is a quasi-projective subvariety of codimension $c$, we
define 
$H^i(V,\QQ)_{\prim}$ to be the kernel of the natural map 
$H^{i}(V,\QQ) \to H^{i+2c}(\PP,\QQ)(c)$, given by repeated cupping with the
hyperplane class. 

In the quasismooth case we can relate $H^i(Y^*,\CC)_{\prim}$ to $H^{i+1}(U,\CC)$
by
using the Poincar\'e residue map. In the non-quasismooth case this is more
subtle.

\begin{proposition} \label{prpPRsing} We have the following:
\begin{enumerate}
\item  Suppose $n= 2$ and $\dim \Sigma= 0$, then 
\begin{eqnarray*}& H^2(Y^*,\QQ)_{\prim}\cong H^3(U,\QQ)(1); \; H^3(Y^*,\QQ)\cong
\QQ(-2)^{\# \Sigma-1} &\\
&\mbox{
and } H^4(Y^*,\QQ)=0.\end{eqnarray*}
\item Suppose $n=3$ and $\dim \Sigma=0$, then 
\begin{eqnarray*} &H^3(Y^*,\QQ)  \cong  H^4(U,\QQ)(1) ;\;  H^4(Y^*,\QQ) \cong
\QQ(-2)& \\ &\mbox{ and }  H^5(Y^*,\QQ)
\cong \QQ(-3)^{\# \Sigma-1}.&\end{eqnarray*}
\item  Suppose $n=3$ and $\dim \Sigma=1$, then 
\[ 0\to H^4(U,\QQ)(1) \to H^3(Y^*,\QQ) \to H^2(\Sigma,\QQ)_{\prim}^*(-3)\to 0 \]
is exact. Moreover
\[ H^4(Y^*,\QQ)\cong H^1(\Sigma,\QQ)^*(-3) \mbox{ and } H^5(Y^*,\QQ) \cong
H^0(\Sigma,\QQ)_{\prim}^*(-3).\]
\end{enumerate}
\end{proposition}

Before proving Proposition \ref{prpPRsing} we shall prove some auxiliary
results.

\begin{proposition}\label{prpThom} We have a Thom-type isomorphism
\begin{equation} \label{eqnThom}T: H^{k}(Y^*,\QQ) \to
H^{k+2}(\PP^*,U,\QQ)(1).\end{equation}
\end{proposition}
\begin{proof}
The map $T$ is induced by the Thom isomorphism on the (punctured) affine cones
over $Y^*$, $\PP^*$ and $U$.
For the precise construction we refer to \cite[Section 2]{DimBet}. 
\end{proof} 
 Consider now the long exact sequence of MHS of the pair $(\PP^*,U)$:
\begin{equation}\label{eqnPair} \ldots \to H^k(\PP^*,U,\QQ) \stackrel{j^*}{\to}
H^k(\PP^*,\QQ)
\stackrel{i^*}{\to} H^k(U,\QQ) \to H^{k+1}(\PP^*,U,\QQ) \to \ldots
\end{equation}

\begin{lemma}\label{lemTopA} We have that
\[ H^k(\PP^*,U,\QQ)\cong H^k(\PP^*,\QQ) \]
for $k>n+2$ and that
\[ H^k(Y^*,\QQ)\cong H^{k+2}(\PP^*,\QQ)(1) \]
for $k>n$.
\end{lemma}
\begin{proof}
Since $U$ is affine we have $H^i(U,\QQ)=0$ for $i\geq n+2$, hence the first
isomorphism follows from sequence (\ref{eqnPair}). The second isomorphism
follows from the Thom isomorphism combined with the first isomorphism.
\end{proof}

Using that $\PP^*$ is a $V$-manifold we can relate $H^k(\PP^*)$ to the
cohomology of $\Sigma$:

\begin{lemma}\label{lemTopB}  
If $\dim \Sigma=0$ then
\[H^i(\PP^*,\QQ) \cong \left\{ \begin{array}{cl} 0 &\mbox{for } i=2n+2\\
H^0(\Sigma,\QQ)_{\prim}^*(-n-1)  & \mbox{for } i=2n+1 \\ H^i(\PP,\QQ) &
\mbox{for } i <
2n+1
\end{array}\right.\]
as MHS and if $\dim \Sigma=1$ then
\[H^i(\PP^*,\QQ) =\left\{ \begin{array}{cl} 0 & \mbox{for } i=2n+2\\
H^0(\Sigma,\QQ)_{\prim}^*(-n-1)
 & \mbox{for } i=2n+1 \\ H^1(\Sigma,\QQ)^*(-n-1) & \mbox{for } i =2n \\
H^2(\Sigma,\QQ)_{\prim}^*(-n-1) &
\mbox{for } i=2n-1\\
 H^i(\PP,\QQ) & \mbox{for } i < 2n-1 \end{array}\right.\]
 as MHS.
\end{lemma}
\begin{proof} We have the Gysin exact sequence
\[ 0 \to H^0_c(\PP^*,\QQ)\to H^0_c(\PP,\QQ)\to H^0_c(\Sigma,\QQ) \to
H^1_c(\PP^*,\QQ) \to \dots
\]
Note that $\PP$ and $\Sigma$ are compact.
If $\dim \Sigma=0$ then it follows immediately from the Gysin sequence that 
\[H^i_c(\PP^*,\QQ) =\left\{ \begin{array}{cl} 0 & i=0\\ H^0(\Sigma,\QQ)_{\prim} 
& i=1
\\ H^i(\PP,\QQ) & i > 1. \end{array}\right.\]

If $\dim \Sigma=1$ it follows that 
\[H^i_c(\PP^*,\QQ) =\left\{ \begin{array}{cl} 0 & i=0\\ H^0(\Sigma,\QQ)_{\prim} 
& i=1
\\ H^1(\Sigma,\QQ) & i =2 \\ H^2(\Sigma,\QQ)_{\prim} & i=3\\
 H^i(\PP,\QQ) & i > 3. \end{array}\right.\]

Since $\PP$ is a V-manifold, the same holds for  $\PP^*$ and we can apply
Poincar\'e duality to obtain the lemma.
\end{proof}

We are now in a position to prove Proposition~\ref{prpPRsing}.

\begin{proof}[{Proof of Proposition~\ref{prpPRsing}}]
Suppose that $n=2$ and $\dim \Sigma=0$. Then we have 
\begin{eqnarray*}  H^3(Y^*,\QQ) & \cong & H^5(\PP^*,U,\QQ)(1) \cong
H^5(\PP^*,\QQ)(1)\\& \cong&
H^0(\Sigma,\QQ)_{\prim}(-2)^*\cong \QQ(-2)^{\# \Sigma-1}.\end{eqnarray*}
The first isomorphism is the  Thom-isomorphism (Proposition~\ref{prpThom}), the
second isomorphism comes from Lemma~\ref{lemTopA}, the third isomorphism comes
from Lemma~\ref{lemTopB} and the fourth isomorphism is immediate. Similarly, one
has $H^4(Y^*,\QQ)\cong H^6(\PP^*,U,\QQ)(1) =0$. To calculate $H^2(Y^*,\QQ)$
consider the
long exact sequence (\ref{eqnPair}) of the pair $(\PP^*,U)$:
\[ \ldots \to H^3(\PP^*,\QQ)\to H^3(U,\QQ) \to   H^4(\PP^*,U,\QQ) \to
H^4(\PP^*,\QQ) \to \ldots .\]
It follows from Lemma~\ref{lemTopB} that $H^3(\PP^*,\QQ)\cong H^3(\PP,\QQ)=0$.
{}From the
same lemma it follows that $H^4(\PP^*,\QQ)\cong H^4(\PP,\QQ)$. Since $U$ is
affine and
of dimension 3, we have that $H^4(U,\QQ)=0$. Finally, the Thom-isomorphism
yields
$H^4(\PP^*,U,\QQ)\cong H^2(Y^*,\QQ)(-1)$. Combining everything gives
\[ 0 \to H^3(U,\QQ) \to H^2(Y^*,\QQ)(-1) \to H^4(\PP,\QQ) \to 0\]
whence $H^3(U,\QQ)(1)\cong H^2(Y^*,\QQ)_{\prim}$.

In the case $n=3$ we  can proceed similarly: combining the Thom
isomorphism with Lemmas~\ref{lemTopA} and~\ref{lemTopB} yields the following
isomorphisms:
\[ H^5(Y^*,\QQ)  \cong H^7(\PP^*,\QQ)(1) \cong H^0(\Sigma,\QQ)_{\prim}^*(-3).\]
If $\dim \Sigma=0$ then
\[ H^4(Y^*,\QQ) \cong H^6(\PP^*,\QQ)(1) \cong H^6(\PP,\QQ)(1)=\QQ(-2)\]
and if $\dim \Sigma=1$ then
\[ H^4(Y^*,\QQ) \cong H^6(\PP^*,\QQ)(1) \cong H^1(\Sigma,\QQ)^*(-3).\]

The calculation of $H^3(Y^*,\QQ)$ is slightly more complicated. We have an exact
sequence 
\[  H^4(\PP^*,\QQ)\to H^4(U,\QQ) \to   H^5(\PP^*,U,\QQ) \to H^5(\PP^*,\QQ) \to
H^5(U,\QQ)=0.\]
{}From Lemma~\ref{lemTopB} it follows that $H^5(\PP^*,\QQ)\cong
H^2(\Sigma,\QQ)_{\prim}^*(-3)$. {}From the same lemma it follows that
$H^4(\PP^*,\QQ)\cong
H^4(\PP,\QQ)$. Since $H^4(\PP,\QQ)\to H^4(U,\QQ)$ is the zero-map, we obtain,
after applying
the Thom-isomorphism, the following short exact sequence
\[ 0 \to H^4(U,\QQ)(1) \to H^3(Y^*,\QQ) \to H^2(\Sigma,\QQ)_{\prim}^*(-3)\to 0.
\]
To finish the proof, note that if $\dim \Sigma=0$ then
$H^0(\Sigma,\QQ)_{\prim}=\QQ^{\# \Sigma-1}$ and $H^2(\Sigma,\QQ)_{\prim}=0$. In
particular, $H^4(U,\QQ)(1)\cong H^3(Y^*,\QQ)$ in this case.
\end{proof}

\begin{remark}
Later on we will show that the contribution of $H^{\bullet}(\Sigma,\QQ)$ to
$H^{\bullet}(Y^*,\QQ)$  is irrelevant for the calculation of $H^4(Y,\QQ)$. 
\end{remark}

\begin{remark} \label{remDeg}
To finish our analysis of $H^n(Y^*,\QQ)$ we give a set of generators for
$H^{n+1}(U,\CC)$. Recall that we have the pole order filtration on
$\Omega^{\bullet}_U$, inducing a filtration on  $H^i(U,\CC)$.

As explained above, the pole filtration on the de Rham complex yields a spectral
sequence. Remark~\ref{rmkminpole} implies that
$P^{1}H^{n+1}(U,\CC)=H^{n+1}(U,\CC)$.
{}From this it follows easily that 
\[ \oplus_{p=0}^{n+1} E^{n+1-p,p}_1  \to H^{n+1}(U,\CC) \]
is surjective.
An easy calculation (the same as in the quasismooth case) shows that
\[ \oplus_{p=0}^{n+1} E^{n+1-p,p}_1 = \oplus_{k=1}^{n+1} R(g)_{d k-w}.\]
The right hand side is finite dimensional  and generates $H^{n+1}(U,\CC)$.
Moreover,
the direct sum decomposition is the same as the direct sum decomposition with
respect to the graded pieces of the polar filtration.
\end{remark}

A summary of our results is the following:

\begin{proposition} \label{prpFil}Suppose $n=3$. Let $C$ be the cokernel of
$H^4(U,\QQ)\to H^4_{\Sigma}(Y,\QQ)$. Suppose $C$ is a pure weight 4 Hodge
structure,
with trivial $(4,0)$ and $(0,4)$-part.
Then the cokernel of
\[ \psi_1: R_{d-w}(g) \to H^4_{\Sigma}(Y,\CC) \]
contains $F^3C_{\CC}$. The cokernel of 
\[ \psi_2: R_{2d-w}(g) \oplus R_{d-w}(g) \to F^2 H^4_{\Sigma}(Y,\CC) \]
contains $F^2C_{\CC}$. Moreover, if $\psi_1$ is surjective, then $C$ has a pure
$(2,2)$-Hodge structure with \[\dim C =\dim \coker( R_{2d-w}(g) \to  
H^4_{\Sigma}(Y^*,\CC)).\]
\end{proposition}

\begin{proof} Since $P^4H^4(U,\CC)$ consists of forms of pole order 0, we have
that
$P^4H^4(U,\CC)$ and $ H^0(\PP,\Omega^4_\PP)$ are isomorphic. Since this group
vanishes we have that
$P^4H^4(U,\CC)=0$. 
Since $F^3 H^4(U,\CC)\subset P^3H^4(U,\CC)$ (by Theorem~\ref{DD}) it follows
that \[
P^3H^4(U,\CC)=\Gr^3_P H^4(U,\CC) \to \Gr^3_F H^4(U,\CC)\] is surjective. Since
$R_{d-w}(g)$
surjects onto $P^3H^4(U,\CC)$ we obtain that $h^{3,1}(C)$ equals the dimension
of
the cokernel of
\[ R_{d-w}(g) \to  \Gr^3_F H^4_{\Sigma}(Y,\CC).\]
Similarly one obtains that $h^{3,1}(C)+h^{2,2}(C)$ equals the dimension of the
cokernel
\[ R_{d-w}(g)\oplus R_{2d-w} (g)\to  F^2 H^4_{\Sigma}(Y,\CC).\]
Finally, if $\psi_1$ is surjective then $0=h^{3,1}(C)=h^{1,3}(C)$. Hence $C$ is
of pure type $(2,2)$ and
\begin{eqnarray*} \dim C_{\CC}=\dim \Gr_F^2 C_{\CC} &=& \dim \coker (R_{2d-w}
(g)\to \Gr_F^2
H^4_\Sigma(Y,\CC))\\
&=& \dim \coker (R_{2d-w} (g)\to  H^4_\Sigma(Y,\CC)).\end{eqnarray*}
\end{proof}

\begin{remark} The above proof could be slightly simplified if
$P^\bullet=F^\bullet$. However, there exist degree 5 surfaces in $\PP^4$ with
one singularity, namely an ordinary double point, such that $F^\bullet\neq
P^\bullet$. See \cite{DimGri}.
\end{remark}

\section{Cohomology of a surface with isolated
ADE-singularities}\label{secSur}
Let $S\subset \PP$ be a surface in a 3-dimensional weighted projective space
given by an equation $g=0$, such that the set $\Sigma$,
the locus where all
partials of $g$ vanish,  is finite and all singularities of $S$ at points of
$\Sigma$ are of type $A_k$, $D_m$ or $E_n$. As usual we set $S^*=S\setminus
\Sigma$. 
We want to calculate $H^2(S,\QQ)_{\prim}$ and for this reason compare it to
a quasismooth surface
$\tilde{S}$ of the same degree as $S$.

\begin{lemma} \label{lemHN} Let $\mu$ be the total Milnor number of $S$. We have
that
$H^i(S,\QQ)$ has a pure Hodge structure of weight $i$ and
\[ h^{p,q}(S)=\left\{\begin{array}{cl}  h^{p,q}(\tilde{S})& \mbox{ if }
(p,q)\neq (1,1) \\
h^{1,1}(\tilde{S})-\mu & \mbox{ if } (p,q)=(1,1).\end{array}\right. \]
\end{lemma}

\begin{proof} 
We first remark that the statement follows from the Lefschetz Hyperplane
Theorem~\ref{prpLHT}
for all $p+q\neq 2,3$. 

Consider the long exact sequence of the pair $(S,S^*)$
\begin{eqnarray*} \ldots  &\to & H^3_\Sigma(S,\QQ)\to H^3(S,\QQ) \to
H^3(S^*,\QQ) \\&\to &
H^4_\Sigma(S,\QQ) \to H^4(S,\QQ)
\to
H^4(S^*,\QQ) \to \ldots \end{eqnarray*}
{}from e.g. \cite[Example 1.9]{DimBet} it follows that $H^3_\Sigma(S,\QQ)=0$. 
For each $p\in \Sigma$ we have that $(S,p)$ is given locally by a weighted
homogeneous equation. In particular, we can find  a small neighborhood $X$ of
$p$
such that $X$ is a cone over a projective curve, and $X^*=X\setminus \{p\}$ is a
$\CC^*$-bundle over this curve. It follows directly from the Leray-spectral
sequence that $H^3(X^*,\QQ)=H^1(\CC^*,\QQ)\otimes  H^2(X,\QQ)
=H^2(X^*,\QQ)(-1)$. {}From the long
exact sequence of the pair $(X,X^*)$ and the fact that $X$ is contractible it
follows
that $H^4_p(S,\QQ)=H^4_p(X,\QQ)=H^3(X^*,\QQ)=\QQ(-2)$.

Using
Proposition~\ref{prpPRsing} the above exact sequence simplifies to
\[ 0 \to H^3(S,\QQ) \to \QQ(-2)^{\# \Sigma-1} \to \QQ(-2)^{\# \Sigma} \to
\QQ(-2)
\to
0.\]
In particular, $H^3(S,\QQ)=0$. The same argument with $\Sigma=\emptyset$ also
shows
$H^3({\tilde S},\QQ)=0$. 
It remains to show that $H^2(S,\QQ)$ has a pure
Hodge structure and to determine the Hodge numbers of $H^2(S,\QQ)$.
 
Let $S'$ be a minimal resolution of the singularities of $S$ that are  contained
in $\Sigma$. The
exceptional locus $E$ consist of a union of smooth rational curves. Each
connected component has an 
intersection matrix of type  $ADE$. We want to apply
Theorem~\ref{thmMV} with $\cZ=\Sigma$ and exceptional locus $E$. Since the
singularities are rational 
we have
$h^1(E,\QQ)=0$. In particular, $H^2(S,\QQ)\hookrightarrow H^2(S',\QQ)$. Since
$H^2(S',\QQ)$ has
pure weight 2 Hodge structure the same holds for  $H^2(S,\QQ)$.

Again using that $S$ has rational singularities it follows that
$h^{2,0}(S)=h^{2,0}(\tilde{S})$ and $h^{0,2}(S)=h^{0,2}(\tilde{S})$ (see e.g.,
\cite[Introduction]{SteAdj}). Since
$e(S)=e(\tilde{S})-\mu$ (e.g., by \cite[Corollary 5.4.4]{Dim}), the lemma
follows. 
\end{proof}

As argued in Section~\ref{secCoh}, we can express the Hodge numbers of
$\tilde{S}$ in terms of the Jacobian ideal of $\tilde{g}$, where $\tilde{g}$ is
an equation for $\tilde{S}$. Let $d=\deg(\tilde{g})$ and $w=\sum w_i$. Let 
$R(\tilde{g})$ be the Jacobian ring of $\tilde{g}$. Then
$h^{2,0}(\tilde{S})=h^{0,2}(\tilde{S})= \dim R(\tilde{g})_{d-w}=\dim
R(\tilde{g})_{3d-w}$
and
$h^{1,1}(\tilde{S})=\dim R(\tilde{g})_{2d-w}$.

We want to calculate $H^2(S,\CC)$ together with the Hodge filtration. 
{}From Proposition~\ref{prpPRsing} it follows that $H^3(U,\CC)(1)\cong
H^2(S,\CC)_{\prim}$. In \cite{SteAdj} it is proven that the Hodge and polar
filtration coincide in this case.

Let $g$ be an equation for $S$ and let $R(g)$ be Jacobian Ring of $S$. Then we
have surjections
\[ R(g)_{d-w}\to H^{2,0}(S,\CC), \;  R(g)_{3d-w} \to H^{0,2}(S,\CC) \]
and 
\[ R(g)_{2d-w} \to H^{1,1}(S,\CC)_{\prim}\]
(cf. the results in Section~\ref{secCoh}, in particular, Remark~\ref{remDeg}).

In \cite{SteAdj} this statement is made more precise. For each singularity
$(S,p)$ let $g_p$ be a local equation and let $R(g_p)$ be the Jacobian ring of
$g_p$. Note that $R(g_p)
$ is naturally isomorphic
to the Milnor algebra of $(S,p)$. Let $\pi_p : R(g) \to R(g_p)$ be the natural
projection. Then
\begin{theorem}[{Steenbrink \cite{SteAdj}}]\label{thmAdj} The Poincar\'e residue
map induces the following isomorphisms
\[ H^{2,0}(S,\CC) \cong R_{d-w}(g) \]
and 
\[ H^{1,1}(S,\CC)_{\prim} \cong \{ f \in
R_{2d-w}(g) \colon f \in \ker(\pi_p) \;\forall p \in \Sigma \}.\] 
\end{theorem}
\begin{proof} This is a reformulation of the main result of \cite{SteAdj}.
We show how this statement can be obtained from the result in \cite{SteAdj}. In
the introduction of \cite{SteAdj} it is argued that $H^{2,0}(S) \cong
R_{d-w}(g)$. 
In Section 5 of \cite{SteAdj} it is moreover shown that $\dim R_{2d-w}(g)=\dim
R_{2d-w}(\tilde{g})(=h^{1,1}(\tilde{S})_{\prim})$. 
As argued in Section~\ref{secCoh} the map
\[ R_{2d-w}(g)\to H^{1,1}(S)_{\prim} \]
is surjective. Using these two facts and $h^{1,1}(S)=h^{1,1}(\tilde{S})-\mu$ we
get that the kernel of
\[ R_{2d-w}(g)\to H^{1,1}(S,\CC)_{\prim} \]
has dimension $\mu$. 

We will now construct a section to this map. Let $j: S\setminus \Sigma \to S$ be
the inclusion. 
Let $\tilde{\Omega}^p_S=j_*\Omega^p_{S\setminus \Sigma}$ and let $\cT$ be the
cokernel of
$d: \Omega^1(S) \to \Omega^2(2S)$. Then $\cT$ is a skyscraper sheaf supported at
$\Sigma$. At each $p\in \Sigma$ we have that
the stalk $\cT_p$ is isomorphic to the Tjurina algebra of $(S,p)$, which is by
definition isomorphic to $R(g_p)$. Since $S$ has only $ADE$ singularities we
have for each $p\in \Sigma$ that the Milnor algebra and the Tjurina algebra of
$(S,p)$ coincide, in particular, $h^0(S,\cT_p)=\mu$.

 Consider the exact
sequence (from \cite[Corollary 17]{SteAdj})
\[ 0\to H^1(S,\tilde{\Omega}_S^1)_{\prim} \to R_{2d-w}(g) \to  H^0(S,\cT) \to
H^2(S, \tilde{\Omega}^1_S) \to 0.\]
As argued in \cite{SteAdj} we have that $ H^2(S,\tilde{\Omega}_S^{1})\subset
H^3(S,\CC)=0$. 

Hence this exact sequence reduces to
\[ 0\to H^1(S,\tilde{\Omega}_S^1)_{\prim} \to R_{2d-w}(g) \to \oplus_{p\in 
\Sigma}
R(g_p)
\to 0.\]
In \cite{SteAdj} it is then argued that
$H^1(S,\tilde{\Omega}_S^1)=H^{1,1}(S,\CC)$. Hence the above map provides the
desired section. (The fact that $H^{1,1}(S) \to R_{2d-w}(g) \to H^{1,1}(S)$ is
actually the identity follows from the construction of the first map in
\cite{SteAdj}.)
\end{proof}

\begin{remark}\label{remunite} Steenbrink's point of view is different from the
approach taken
by Dimca. In the previous section we constructed a surjection from $R_{2d-w}(g)$
onto 
$H^{1,1}(S,\CC)$, whereas Steenbrink constructs an injection from
$H^{1,1}(S,\CC)$ to  
$R_{2d-w}(g)$,
which is a section of the former map.

To unite the two approaches we can do the following. Let $\mu$ be the total
Milnor number of $S$. Fix $\mu$ polynomials $h_1,\dots,h_\mu$ of degree $2d-w$
such that their image spans $\oplus_{p\in \Sigma} R(g_p)$. Set
$\tilde{R}(g):=R(g)/(h_1,\dots,h_\mu)$. Then $H^{2,0}(Y,\CC)\cong
\tilde{R}_{d-w}(g)$ and $H^{1,1}(Y,\CC)\cong \tilde{R}_{2d-w}(g)$.   
\end{remark}

\begin{remark}
Suppose $p\in \Sigma$ has a non-trivial stabilizer group, i.e., 
$\tilde{p}:=(x_0,x_1,x_2,x_3)$ is a lift of $p$ to $\CC^4$ 
and the stabilizer subgroup $G_p \subset \CC^*$ of $\tilde{p}$  is non-trivial.

Without loss of generality we can assume that   $\tilde{p}=(1,\alpha,0,0)$.
Suppose $f(x_0,x_1,x_2,x_3)$ is a defining polynomial for $S$. Let
$g(x_1,x_2,x_3)=f(1,x_1+\alpha,x_2,x_3)$. If $G_p$ consists of one element then
the Milnor algebra of $(S,p)$ equals
$\CC\{x_1,x_2,x_3\}/(g_{x_1},g_{x_2},g_{x_3})$. However, if $\# G_p>1$ 
then the Milnor algebra of $(S,p)$ equals
\[ \left(\CC\{x_1,x_2,x_3\}/(g_{x_1},g_{x_2},g_{x_3})\right)^{G_p}.\]
\end{remark}

\section{Calculation of $H^4_{\Sigma}(Y,\CC)$, local information}\label{secLoc}
In this and the following section we assume that $Y$ is an admissible
hypersurface in a weighted projective space $ \PP(w_0,\dots,w_4)$ (cf. the
Introduction) given by
$f=0$. 
Let $\Sigma\subset  \PP(w_0,\dots,w_4)$ be the locus where all partials of $f$
vanish. 

Since $Y$ is admissible we can find for every  $p\in \Sigma$ a weighted
homogeneous polynomial 
$g_p$ (with weights $w_{1,p}$, $w_{2,p}$,
$w_{3,p}$, $w_{4,p}$ and degree $d_p$)
such that  
\begin{enumerate}
\item  $(Y,p)$ is contact equivalent to  $(\{g_p=0\},0)\subset (\CC^4,0)$;
\item  the surface $S:=\{g_p=0\}\subset \PP(w_{1,p},w_{2,p},w_{3,p},w_{4,p})$
 has finitely many $ADE$-singularities.
\end{enumerate}

\begin{remark}
The conditions on the singularities of $Y$ are very mild. For example in the
case of elliptic threefolds we considered hypersurfaces of the form
$y^2=x^3+Px+Q$, with $(P,Q)\in  \CC[z_0,z_1,z_2]_{4n}\times
\CC[z_0,z_1,z_2]_{6n}$. For fixed $n$ the locus where  the  conditions on the
singularities are not satisfied has a large codimension.
E.g., in the isolated singularity case  the most frequently occuring
singularities such as $ADE$ threefold singularities are all weighted homogeneous
singularities.
\end{remark}

\begin{remark}\label{remCont}
Recall that two singularities $(\{f_1=0\},0)$ and $(\{f_2=0\},0)$ are contact
equivalent if and only if
\[ \CC\{x_1,\dots,x_n \}/(f_1)\cong \CC\{x_1,\dots,x_n \}/(f_2). \]
If $f_1$ (and $f_2$) are isolated singularities then $f_1$ and $f_2$ are 
contact equivalent if and only if their Milnor algebras are isomorphic. If we
assume that $f_1$ is weighted homogeneous then, by the Euler formula, we get
$f_1+J(f_1)=J(f_1)$, hence the Tjurina algebra and the Milnor algebra of $f_1$
are isomorphic. 

It turns out that if $f_2$ is isolated and contact equivalent to a weighted
homogeneous singularity $f_1$ then it is also right equivalent to $f_1$, and
hence the Tjurina algebra of $f_2$ is isomorphic to the Tjurina algebra of
$f_1$.
This implies that in the isolated case we could reword our condition on $(Y,p)$
by saying that the Milnor number and the Tjurina number of $(Y,p)$ coincide.
(Details of this reasoning can be found in \cite[Theorem 7.42]{DimReal} and
\cite[Section 9.1]{JongPfi}.)

For non-isolated singularities we are not aware of such a simple reformulation.
\end{remark}
\begin{remark}
Note that the surface $S$ satisfies the hypothesis of the
previous section.
We define $S^*=S\setminus \Sigma_p$ where $\Sigma_p$ is the locus where all the 
partials of $g_p$ 
vanish.
Let $X\subset \CC^4$ be the zero set of $g_p$, i.e. the affine cone over the
surface $S$. 
\end{remark}

\begin{lemma}\label{lemComp}
\[ H^i_p(Y,\QQ) \cong H^i_0(X,\QQ).\]
\end{lemma}
\begin{proof}
This follows directly from the definition of contact equivalence.
\end{proof}
Let $\Sigma' $ be the singular locus of $X$ and set $X^*=X\setminus \{0\}$. In
this
section we relate $H^\bullet_{0}(X,\QQ)$ to $H^\bullet(S,\QQ)$.

\begin{lemma}\label{lemCone} 
For $i>1$ we have isomorphisms
\[ H^i_{0}(X,\QQ)\cong H^{i-1} (X^*,\QQ). \] 
Moreover,
\[ H^i_{0}(X,\QQ)=0 \]
for $i=0,1$.
\end{lemma}

\begin{proof}
 Since $X$ is the affine cone over $S\subset
\PP(w_{1,p},w_{2,p},w_{3,p},w_{4,p})$ 
it is contractible and hence $H^i(X,\QQ)=0$ 
for
$i>0$. The long exact sequence of the pair $(X,X^*)$ therefore yields an
isomorphism
 \[ H^i_{0}(X,\QQ)\cong H^{i-1} (X^*,\QQ)\] 
for $i>1$. Clearly, the natural map
\[ H^0(X,\QQ)\to H^0(X^*,\QQ)\]
is an isomorphism. Since $H^1(X,\QQ)=0$ the same sequence gives that both
$H^0_{0}(X,\QQ)$ and $H^1_{0}(X,\QQ)
$ vanish.
 \end{proof}

The cone $X^*$ is a $\CC^*$-fibration over $S$. Recall from Section~\ref{secSur}
that  $H^i(S,\QQ)$ vanishes unless $i=0,2,4$ and that $H^0(S,\QQ)=\QQ$,
$H^4(S)=\QQ(-2)$. The
Hodge structure on $H^2(S,\QQ)$ can be calculated by Theorem~\ref{thmAdj}. This
enables us to calculate the Hodge structure of $H^\bullet_0(X,\QQ)$.
  \begin{proposition}\label{prpIsol} We have that
\[ H^i_{0}(X,\QQ) = \left\{ \begin{array}{cl} 
 H^2(S,\QQ)_{\prim} & \mbox{for } i=3 \\
 H^2(S,\QQ)_{\prim}(-1)  & \mbox{for } i=4 \\
 \QQ(-3) & \mbox{for } i=6 \\
 0 & \mbox{otherwise.}

\end{array}\right.\] 
\end{proposition}
\begin{proof}

Consider the $E_2$ part of the Leray spectral sequence for $X^*\to S$: 
  \[ 
 \begin{array}{c||c|c|c|c|c|}
 H^1(\CC^*,\QQ) &\QQ(-1) & 0 & H^2(S,\QQ)(-1) &0 & \QQ(-3) \\ 
 H^0(\CC^*,\QQ) &\QQ & 0 & H^2(S,\QQ) & 0 & \QQ(-2) \\ 
 \hline
 & H^0(S,\QQ)& H^1(S,\QQ) & H^2(S,\QQ) & H^3(S,\QQ) & H^4(S,\QQ)
 \end{array}
 \]
The only possible  non-zero differentials are the maps $\QQ(-1)\to H^2(S,\QQ)$ 
and
$H^2(S,\QQ)(-1) \to \QQ(-2)$. We will show below that these maps are actually
injective, respectively
surjective. Assuming this for the moment it follows that  the $E_3$-terms equals
  \[ 
 \begin{array}{c||c|c|c|c|c|}
 H^1(\CC^* ) &0 & 0 & H^2(S,\QQ)_{\prim}(-1) & 0 & \QQ(-3) \\ 
 H^0(\CC^* ) &\QQ & 0 & H^2(S,\QQ)_{\prim} & 0& 0 \\ 
 \hline
 & H^0(S )& H^1(S ) & H^2(S ) & H^3(S ) & H^4(S )
 \end{array}
 \]
and the spectral sequence degenerates at $E_3$. Hence $H^i(X^*,\QQ)\cong
\oplus_j E_3^{i-j,j}$
 and thus 
  \[ H^i(X^*,\QQ) = \left\{ \begin{array}{cl} 
 \QQ & \mbox{for } i=0 \\
 0 & \mbox{for } i=1 \\
 H^2(S,\QQ)_{\prim} & \mbox{for } i=2 \\
 H^2(S,\QQ)_{\prim}(-1)  & \mbox{for } i=3 \\
  0& \mbox{for } i=4 \\
 \QQ(-3) & \mbox{for } i=5. \\
\end{array}\right.\]

By Lemma \ref{lemCone} we have $H^i_{0}(X,\QQ)=H^{i-1}(X^*,\QQ)$ for $i>1$ and
thus we
obtain the
proposition.

It remains to show that the differential $\QQ(-1)\to H^2(S,\QQ)$ is injective 
and
that the differential $H^2(S,\QQ)(-1) \to \QQ(-2)$ is surjective.

Let $\tilde{X}$ be the blow-up of $X$ at $0$. Then $\tilde{X}$ is  a
$\CC$-fibration over $S$.
Note that $S$ admits Poincar\'e duality (a consequence of Lemma~\ref{lemHN}).
Using that $H^i_c(\CC^*,\ZZ)=0$ for $i\neq 1$ it follows that the Leray-Spectral
sequence (for cohomology with compact support)
 associated with $\tilde{X} \to S$ degenerates at $E_2$ and we get that
$H^{6-i}_c(\tilde{X},\QQ)\cong H^i(S,\QQ)(-1)$. Similarly, we get that
$H^i(\tilde{X},\QQ)=H^i(S,\QQ)$.

Let $E\subset\tilde{X}$ be the exceptional divisor. Then $E\cong S$ and
$\tilde{X}\setminus E=X^*$.
Consider the following part of the Gysin exact sequence:
\begin{eqnarray*}  H^1(E,\QQ)=0 &\to& H^2_c(X^*,\QQ) \to H^2_c(\tilde{X},\QQ)
\to H^2(E,\QQ) \\ &\to&
H^3_c(X^*,\QQ) \to
H^3_c(\tilde{X},\QQ) = 0. \end{eqnarray*}
The map $H^2_c(\tilde{X},\QQ)\to H^2(E,\QQ)$ is induced by a map from integral
cohomology.  
Let $h\in H^2(E,\ZZ)$ be the hyperplane class. {}From the Leray spectral
sequence
it follows that $H^2_c(\tilde{X},\ZZ)=H^0(E,\ZZ)\otimes H^2_c(\CC,\ZZ)$. Let
$h_1\in H^2_c(\tilde{X},\ZZ)$  be $[E]$ times a generator of $H^2_c(\CC,\ZZ)$. 
Let $\iota:E\to \tilde{X}$ be the inclusion. Then it is easy to see that
$\iota^*(h_1)=-h$.
Hence the map $\iota^*$ is not constant and since
$h^2_c(\tilde{X},\QQ)=h^4(S,\QQ)=1$ it
follows that $\iota^*$ is injective. {}From the Gysin exact sequence it follows
that $H^2_c(X^*,\QQ)=0$  and that $h^3_c(X^*)=h^2(E)-1$.
Assume for the moment that $X^*$ is smooth, i.e., $E$ is quasismooth. Using
Poincar\'e duality we get that $h^3(X^*)=h^2(E)-1$. Since $H^3(X^*,\QQ)$ equals
\[\ker
(H^2(E,\QQ)(-1)\to H^4(\tilde{X},\QQ))=\ker( H^2(S,\QQ)(-1) \to \QQ(-2))\] it
follows
that the
differential $H^2(S,\QQ)(-1)\to \QQ(-2)$ is surjective.

For the other differential we can proceed similarly:
\begin{eqnarray*} H^3(E,\QQ)=0 &\to&  H^4_c({X}^*,\QQ) \to H^4_c(\tilde{X},\QQ)
\to H^4(E,\QQ)\\& \to&
H^5_c(X^*,\QQ) \to
H^5_c(\tilde{X},\QQ)=0 .\end{eqnarray*}
The map $H^4_c(\tilde{X},\QQ)\to H^4(E,\QQ)$ is again induced by a map on
integral
cohomology, and the class of $h$ times a generator of $H^2_c(\CC,\ZZ)$ is mapped
to a nonzero multiple of a generator of $H^4(E,\ZZ)$. This implies that
$h^4_c(X^*)=h^4_c(\tilde{X})-h^4(E)=h^2(E)-1$. Using Poincar\'e duality we get
that the differential $\QQ(-1) \to H^2(S,\QQ)$ is injective, provided that $S$
is
quasismooth.

If $S$ is not quasismooth then we can find a family of quasismooth surfaces
$S_{\lambda}$ degenerating to $S$ for $\lambda=0$. Now for $\lambda\neq 0$, we
have that the differential 
\[ \QQ(-1) \to H^2(S_{\lambda},\QQ)\] is induced by a non-zero map
$H^2(\tilde{X}_{\lambda},\ZZ) \to H^2(E_{\lambda},\ZZ)$. Let $h_\lambda$ be a
family of
generators of $H^2(E_{\lambda},\ZZ)$ and let $h'_{\lambda}$ be a family of
generators of $H^2(\tilde{X}_{\lambda},\ZZ)$. Then $h'_{\lambda}$ is mapped to
$-h_{\lambda}$.
By taking the limit $\lambda \to 0$, we see that $h'_0$ is  mapped to $-h_0$,
hence
$H^2(\tilde{X}_0,\QQ)\to H^2(E,\QQ)$ is injective, and from this it follows that
$\QQ(-1)\to
H^2(S,\QQ)$ is injective. A similar argument shows that also $H^2(S,\QQ)\to 
\QQ(-2)$ is surjective. This finishes the proof.
\end{proof}

The following proposition will be useful for our purposes
\begin{proposition}\label{propLocCoh} Let $Y,p,d_p$ be as above. Let
$w_p=w_{1,p}+w_{2,p}+w_{3,p}+w_{4,p}$.  
 Then
$H^4_p(Y,\QQ)$ has a pure weight $4$ Hodge structure without $(0,4)$ and
$(4,0)$-component. We have
  \[ F^3 H^4_p(Y,\CC) \cong \tilde{R}_{d_p-w_p}(g_p) \]
and
 \[ F^2 H^4_p(Y,\CC)/F^3 H^4_p(Y,\CC) \cong \tilde{R}_{2d_p-w_p}(g_p)\]
 where $\tilde{R}$ is obtained from $R$ as explained in Remark~\ref{remunite}.
\end{proposition}
\begin{proof}
This is a combination of Lemma~\ref{lemComp}, Proposition~\ref{prpIsol} and
Theorem~\ref{thmAdj}. 
\end{proof}

\begin{proposition} \label{prpTransSurf}
Let $(Y,p)$ be a transversal  $ADE$ surface singularity. Then 
$H^6_p(Y,\QQ)=\QQ(-3)$ and $H^i_p(Y)=0$ for
$i\neq 6$.
\end{proposition}

\begin{proof}
  For simplicity we assume that $(Y,p)$ is an $A_k$-singularity. Using
Lemma~\ref{lemComp}  it suffices to 
prove the statement for $(Y,p)$ given by 
\[ x_1^2+x_2^2+x_3^{k+1} =0 .\]
This equation defines a surface $S \subset 
\PP(k+1,k+1,2,1)$ of degree $2k+2$ with an isolated $A_k$
singularity in $(0:0:0:1)$.

{}From Lemma~\ref{lemComp} and Proposition~\ref{prpIsol} it follows that it
suffices to prove that
$H^2(S,\QQ)_{\prim}=0$.
We start by calculating $h^2(\tilde{S})$ for a quasismooth surface $\tilde{S}$
of the same degree, e.g.,
$\tilde{g}:=x_1^2+x_2^2+x_3^{k+1}+x_4^{2k+2}=0$. This can be done by calculating
the dimension of several graded pieces of the Jacobian ring of $\tilde{Y}$.
The sum of the weights equals $2k+5$, hence we are interested in
$h^{2,0}(\tilde{S})=\dim R(\tilde{g})_{-3}=0$,
$h^{0,2}(\tilde{S})=R(\tilde{g})_{4k+1}= 0$ and 
\[ h^{1,1}(\tilde{S})=\dim R(\tilde{g})_{2k-1}= \dim \spa \{ [x_3^ix_4^j] \colon
2i+j=2k-1\}=k.\]
Hence $h^2(\tilde{S})_{\prim}=k$.
Since
$\mu(Y,p)=k$, we get $h^2(S)_{\prim}=h^2(\tilde{S})_{\prim}-\mu(Y,p)=0$. This
finishes the $A_k$ case.

For $D_m,E_n$ singularities one can proceed similarly.
\end{proof}

\section{Glueing local information}\label{secGlue}
Let $\PP$ be a four dimensional weighted projective space and let $Y\subset \PP$
be a hypersurface, given by $f=0$. Let $\Sigma$ be the locus where all the
partials of $f$ vanish. We assume the usual conditions, i.e., 
$\Sigma\cap \PP_{\sing}=\emptyset$,
$\dim \Sigma \leq 1$ and that 
at a general point of any one dimensional component of $\Sigma$ we have a
transversal  $ADE$ surface singularity.
Finally, let $\cP\subset \Sigma$ be the set of points $p\in \Sigma$ such that
$(Y,p)$ is not a
transversal  $ADE$ surface singularity.

We want to use the previous section to relate $H^4(Y,\QQ)_{\prim}$ to the 
cokernel of $H^4(U,\QQ)(1) \to \oplus_{p\in \cP} H^4_p(Y,\QQ)$. In this section
all considerations are topological. For this reason we work with $\QQ$
coefficients and use $H^i(\cdot)$ as   shorthand for $H^i(\cdot,\QQ)$.

For each point $p\in \cP$, fix a small contractible neighborhood $U_p\subset
\Sigma$.
Let $\Sigma_1:=\Sigma \setminus \cup_{p\in \cP} U_p$ be the complement of the
$U_p$. Note that $\Sigma_1$ is a closed Riemann surface with boundary embedded
in $\PP$.

\begin{lemma} \label{lemGpt}
We have that \[ H^4_{\Sigma_1}(Y)\cong H^2(\Sigma_1)^*(-3), \;
H^5_{\Sigma_1}(Y)\cong H^1(\Sigma_1)^*(-3)\] and  
\[ H^6_{\Sigma_1}(Y)\cong  H^0(\Sigma_1)^*(-3). \]
\end{lemma}

\begin{proof} 
Take a finite open covering  $\cU:=\{V_i\}$ of $\Sigma_1$ such that each $V_i$
is homeomorphic to a disc with boundary $S^1$, in particular each $V_i$ is
contractible. 
Let $D_i=\overline{V_i}$ be the closure in the complex topology. 
It is easy to show that we can find such a covering with the property
that each intersection $D_{i_1} \cap D_{i_1}\cap \dots\cap D_{i_k}$ is
empty or contractible.

We now proceed by induction. If $\# \cU=1$, then $\Sigma_1$ is contractible.
Hence
$H^0(\Sigma_1)=\QQ$ and all other cohomology groups of $\Sigma_1$ vanish. 
In this case we have a deformation retract $(Y,Y\setminus \Sigma_1)$ to $(Y',Y'
\setminus
\{ p \})$ where $(Y',p)$ is a transversal  $ADE$ surface singularity. {}From
this
it follows that $H^i_{\Sigma_1}(Y)\cong H^i_p(Y')$. {}From
Proposition~\ref{prpTransSurf} it follows that $H^6_p(Y')=\QQ(-3)$ and all other
local cohomology groups vanish. Hence the statement is true in this case.

Assume now $\# \cU=k$, let $\Sigma_0=\cup_{1\leq i \leq k-1} D_i$.
We have two Mayer-Vietoris sequences (one is dual to the usual Mayer-Vietoris 
sequence, the other is Mayer-Vietoris for cohomology with support), namely 
\begin{SMALL}
\[ \xymatrix{  H^{i}(D_k \cap \Sigma_0)^* \ar[r]\ar[d]^{\sim} &
H^{i}(D_k)^*\oplus H^{i}(\Sigma_0)^* \ar[r]\ar[d]^{\sim} &
H^{i}(\Sigma_1)^* \ar[r]\ar[d] &  H^{i-1}(D_k \cap \Sigma_0)^*
\ar[d]^{\sim} \\
 H^{6-i}_{D_k \cap \Sigma_0}(Y)(3) \ar[r] & H^{6-i}_{D_k}(Y)(3) \oplus
H^{6-i}_{\Sigma_0}(Y)(3) \ar[r]& H^{6-i}_{\Sigma_1}(Y)(3)\ar[r] 
&H^{6-(i-1)}_{D_k \cap
\Sigma_0}(Y)(3) } \]
\end{SMALL}
 The first two vertical maps are isomorphisms by the induction hypothesis.
{}From
the five-lemma it follows 
that $\dim H^{i}(\Sigma)= \dim H^{6-i}_\Sigma(Y)$, which
yields the lemma.
\end{proof}

\begin{lemma} \label{lemLocH0} We have that
\[ H^6_{\Sigma}(Y)\cong H^0(\Sigma)^*(-3) \mbox{ and } H^5_{\Sigma}(Y) \cong
H^1(\Sigma)^*(-3). \]
\end{lemma}

\begin{proof} 
Let $D_p=\overline{U_p}$. Using that  $D_p$ is contractible we have that
$H^i_{D_p}(Y)\cong H^i_p(Y)$. {}From Proposition~\ref{prpIsol} it follows that
$H^6_p(Y)=\QQ(-3)$ and also that $H^5_p(Y)=0$. 

Let $\Sigma_2=\cup \overline{U_p}$. Since $\overline{U_p}$ is contractible we
have that $H^1(\Sigma_2)=0$ and $H^5_{\Sigma_2}(Y)=\oplus H^5_p(Y)=0$. In a
similar way we get $H^6_{\Sigma_2}(Y)=\QQ(-3)^{\#\cP}=H^0(\Sigma_2)^*(-3)$.

Along $D:=\Sigma_1\cap \Sigma_2$, which is union of circles, we have transversal
ADE surface singularities. A reasoning as in Lemma~\ref{lemGpt} shows that
$H^5_D(Y) \cong H^1(Y)^*$ and $H^6_D(Y)\cong H^0(D)^*$.

As in the previous lemma we can consider the two Mayer-Vietories sequences (the
vertical arrows are isomorphisms by either the above discussion or using
Lemma~\ref{lemGpt})
\begin{small}
\[ \xymatrix{  H^{1}(D)^* \ar[r]\ar[d]^{\sim} &
H^{1}(\Sigma_1)^*\oplus H^{1}(\Sigma_2)^* \ar[r]\ar[d]^{\sim} &
H^{1}(\Sigma)^*\ar[r]\ar[d] &  \dots
\\
 H^{5}_{D}(Y)(3) \ar[r] & H^{5}_{\Sigma_1}(Y)(3) \oplus
H^{5}_{\Sigma_2}(Y)(3) \ar[r]& H^{5}_{\Sigma}(Y)(3)\ar[r]  &\dots  } \]
\[ \xymatrix{  \dots \ar[r] &  H^{0}(D)^*
\ar[d]^{\sim}\ar[r] & H^{0}(\Sigma_1)^*\oplus H^{0}(\Sigma_2)^*
\ar[r]\ar[d]^{\sim} &
H^{0}(\Sigma)^* \ar[r]\ar[d] & 0
\\
\dots  \ar[r]&H^6_D(Y)(3)\ar[r] & H^{6}_{\Sigma_1}(Y)(3) \oplus
H^{6}_{\Sigma_2}(Y)(3) \ar[r]& H^{6}_{\Sigma}(Y)(3)\ar[r]  &0 } \]
 \end{small}
  An application of the five-lemma yields the proof.
\end{proof}

\begin{lemma}\label{lemH5van} Suppose $\dim \Sigma=1$. Then $H^5(Y)=0$ and 
$H^4(Y^*)\to
H^5_\Sigma(Y)$ is an isomorphism.
\end{lemma}

\begin{proof}
Consider the exact sequence of the pair $(Y,Y^*)$
\begin{eqnarray*} H^4(Y) &\to& H^4(Y^*) \to H^5_{\Sigma}(Y) \to H^5(Y)\\ & \to&
H^5(Y^*) \to H^6_{\Sigma}(Y) \to H^6(Y) \to H^6(Y^*)=0.\end{eqnarray*} 

Note that it follows from Proposition~\ref{prpPRsing} that 
$H^5(Y^*)=H^0(\Sigma)_{\prim}^*(-3)$. Using Lemma~\ref{lemLocH0} it follows that
$h^5(Y^*)=  h^6_{\Sigma}(Y) -h^6(Y)$, hence  the map 
$H^5(Y^*)\to H^6_{\Sigma}(Y)$ is injective.

{}From Proposition~\ref{prpPRsing} it follows that $H^4(Y^*)$ is isomorphic to
$ 
H^1(\Sigma)^*(-3)$. {}From Lemma~\ref{lemLocH0} it follows that
$H^5_{\Sigma}(Y)$
is isomorphic to $H^1(\Sigma)^*(-3)$. Hence $H^4(Y^*)$ and $H^5_{\Sigma}(Y)$
have the same dimension.

Note that the possible Hodge weights of $H^4(Y^*)\cong H^1(\Sigma)^*(-3)$ are
$5$ and $6$, where $H^4(Y)$ has Hodge weights at most 4 \cite[Theorem 5.39]{PS}.
Hence $H^4(Y) \to H^4(Y^*)$ is the zero-map, $H^4(Y^*)\cong H^5_{\Sigma}(Y)$ and
$H^5(Y)=0$.
\end{proof}

\begin{theorem} \label{thmCoKer}
 We have that 
\[H^4(Y)_{\prim} = \coker (H^4(U)(1) \to \oplus_{p\in \cP} H^4_{p}(Y)).\]
\end{theorem}

\begin{proof}
Suppose first that $\dim \Sigma=0$. Then $\cP=\Sigma$. 

Consider the exact
sequence
\[  H^3(Y^*) \to H^4_\Sigma(Y) \to H^4(Y) \to H^4(Y^*).\]
{}From Proposition~\ref{prpPRsing} it follows $H^4(Y^*)_{\prim}=0$ and
$H^3(Y^*)\cong H^4(U)(1)$, hence we have an exact sequence
\[ H^4(U)(1) \to H^4_\Sigma(Y) \to H^4(Y)_{\prim} \to 0.\]
This proves the case $\dim \Sigma=0$.

Suppose that $\dim \Sigma=1$.
Consider the diagram (where both the horizontal and the vertical sequence are 
exact)
\[
\xymatrix{0 \ar[r] & H^4(U)(1) \ar[r] & H^3(Y^*) \ar[r] \ar[d] &
H^2(\Sigma)^*(-3)_{\prim} \ar[r] & 0 \\ && H^4_\Sigma(Y) \ar[d] \\ &&
H^4(Y)\ar[d]\\&&0 .}
\]
The horizontal sequence comes from Proposition~\ref{prpPRsing}, the vertical
sequence is part of the 
long exact sequence of the pair $(Y,Y^*)$. {}From Lemma~\ref{lemH5van} it
follows
that $H^4_{\Sigma}(Y)\to H^4(Y)$ is surjective.

We start by constructing a map $ H^4_\Sigma(Y) \to H^2(\Sigma)^*(-3)$: 
let $\tilde{Y}$ be a resolution of all singularities contained in $\Sigma$ of
$Y$. Let $E$ be the exceptional divisor.
Then there is a natural map $H^2(\Sigma)\to H^2(E)$. Since $\tilde{Y}$ is smooth
we have that $H^i_E(\tilde{Y})=H^{6-i}(E)^*(-3)$. The resolution $(\tilde{Y},E)
\to (Y,\Sigma)$ induces a natural map
$H^i_{\Sigma}(Y) \to H^i_E(\tilde{Y})$. Composing the maps as follows
\[ H^4_{\Sigma}(Y) \to H^4_E(\tilde{Y}) \cong H^2(E)^*(-3) \to
H^2(\Sigma)^*(-3)\]
yields a map $H^4_{\Sigma}(Y) \to H^2(\Sigma)^*(-3)$. It is easy to check that
the composition $H^3(Y^*)\to H^4_{\Sigma}(Y)\to H^2(\Sigma)^*(-3)_{\prim}$ is
the same map as the map $H^3(Y^*)\to H^2(\Sigma)^*(-3)$ in the above diagram.

Let $K$ be the kernel of the map $H^4_{\Sigma}(Y)_{\prim}\to
H^2(\Sigma)_{\prim}^*(-3)$.
The above diagram shows that 
\[ H^4(Y)_{\prim}= \coker H^3(Y^*)\to H^4_{\Sigma}(Y)_{\prim} = \coker H^4(U)(1)
\to K.\]
The final equality is a consequence of the snake lemma.

Hence it remains to show that
\[ K  \cong \oplus_{p \in \cP}
H^4_p(Y). \]

Let $\Sigma_2:= \cup \overline{U_p}$ and $D=\Sigma_1 \cap \Sigma_2$. Note that
$D$ is a union of circles.
Consider the Mayer-Vietoris sequence
\begin{eqnarray*} H^4_D(Y) &\to & H^4_{\Sigma_1}(Y) \oplus H^4_{\Sigma_2} (Y)\to
\\&\to& H^4_\Sigma (Y)\to H^5_D(Y) \to H^5_{\Sigma_1}(Y) \oplus
H^5_{\Sigma_2}(Y) \to H^5_\Sigma(Y) \end{eqnarray*}
Note that $H^5_\Sigma(Y)= H^1(\Sigma)^*(-3)$ by Lemma~\ref{lemLocH0}. Note also
that that $H^5_{\Sigma_2}(Y)=H^1(\Sigma_2)^*(-3)$ by a reasoning similar to the
one in the proof of Lemma~\ref{lemGpt}. Since $H^1(\Sigma_2)=0$ it follows that
$H^5_{\Sigma_2}=0$. 

Since we have transversal $ADE$ singularities  along $D$ and $\Sigma_1$ this
sequence becomes (after tensoring with $\QQ(3)$)
\begin{eqnarray*}  0=H^2(D)^*  &\to & H^2(\Sigma_1)^* \oplus
H^4_{\Sigma_2}(Y)(3) \to \\
&\to& H^4_\Sigma(Y)(3) \to H^1(D)^* \to H^1(\Sigma_1) \to H^1(\Sigma)^* \to
\dots \end{eqnarray*}
Since  $\Sigma_1$ is a deformation retract of $\Sigma \setminus \cP$ we obtain
the following exact sequence: (dualized sequence of the pair $(\Sigma,
\Sigma\setminus \cP)$)
\[ 0\to H^2(\Sigma_1)^* \to H^2(\Sigma)^* \to \oplus_{p\in \cP}  H^2_p(\Sigma)^*
\to H^1(\Sigma_1)^* \to H^1(\Sigma)^*\]

This yields a diagram
\begin{Small}
\[\xymatrix{
 0 \ar[r]  & H^2(\Sigma_1)^* \oplus H^4_{\Sigma_2}(Y)(3)
\ar[r]^{\varphi_1}\ar[d]  & H^4_\Sigma(Y)(3) \ar[r]\ar[d] & H^1(D)^*\ar[r]
\ar[d]& H^1(\Sigma_1)^* \ar[d]^=\\
0\ar[r]& H^2(\Sigma_1)^*\ar[r]^{\varphi_2} &  H^2(\Sigma)^* \ar[r] 
&\oplus_{p\in \cP}  H^2_p(\Sigma)^*  \ar[r] & H^1(\Sigma_1)^*  \\
 }\]\end{Small}
 
Here, the map $H^4_{\Sigma_2}(Y)(3)\to H^2(\Sigma)^*$ is the unique map, making
this
diagram commutative.

Using that $g_p=0$ is weighted homogeneous we get that $(\Sigma,p)$ is locally a
set of $m$ lines through $p$. In particular, ${U_p} \setminus \{p\}$ can be
retracted to
$\overline{U_p}\cap \Sigma_1$. Taking direct sums over all $p\in \cP$ this shows
that $H^i(\Sigma_2 \setminus \cP) \cong H^i(D)$. 
 Since for each $p\in \cP$ we have that $U_p$ is contractible we get a natural
isomorphism
 \[ H^{i+1}_{\cP}(\Sigma) \cong  H^i(\Sigma_2 \setminus \cP) \cong
H^i(D). \]
 Hence the above diagram simplifies to 
 \[\xymatrix{
 0\ar[r] &  H^2(\Sigma_1)^* \oplus H^4_{\Sigma_2}(Y) \ar[r]^{\varphi_1}\ar[d] &
H^4_\Sigma(Y) \ar[r]   \ar[d]&\coker{\varphi_1}\ar[r]\ar[d]^{\sim} &0\\
 0 \ar[r] & H^2(\Sigma_1)^*\ar[r]^{\varphi_2} &  H^2(\Sigma) \ar[r] 
 & \coker \varphi_2\ar[r]& 0
}\]
(The main point here is that $\coker \varphi_1\cong \coker \varphi_2$.)
 {}From this diagram it follows that $ H^4_{\Sigma_2}(Y)=\ker (H^4_\Sigma(Y) 
\to
H^2(\Sigma)^*) = \ker (H^4_\Sigma(Y)_{\prim}  \to H^2(\Sigma)_{\prim}^*) $. 
 
Since the $D_p:=\overline{U_p}$ are contractible,
there exists a deformation retract from
$Y\setminus \Sigma_2$ to $Y\setminus \cP$,
hence $H^4_{\Sigma_2}(Y) \cong H^4_{\cP}(Y)$, which yields the proof.
\end{proof}

\section{Method for calculating $MW(\pi)$} \label{secMethod}

In this section we present a method  to
calculate
the Mordell-Weil rank of a general elliptic threefold.

We start by identifying the set $\Sigma$ and a finite  subset $\cP'$
containing the set $\cP$ (cf. the previous section.)

\begin{proposition}\label{prpDiscLocus}
Suppose we have a threefold $Y\subset \PP(2n,3n,1,1,1)$ defined by the vanishing
of 
$g:=-y^2+x^3+Px+Q$, where $P$ and $Q$ are homogeneous polynomials in
$z_0,z_1,z_2$ of degree $4n$ and $6n$. Suppose $Y$ is minimal.

Let $\Delta$ be the curve defined by $4P^3+27Q^2=0$ and $\Delta_1$ be the
underlying reduced curve.
Let $\psi  : \PP(2n,3n,1,1,1) \to \PP^2$ be the projection onto the plane
$x=y=0$.
Take $\cP$ to be the set defined in Section~\ref{secGlue}. Then $\psi(\cP)$ is
contained in $ \Delta_{1,\sing} \cup \cQ_1 \cup \cQ_2$ 
where 
\[ \cQ_1 := \{ q \in \Delta_{1,\sm} \colon q \mbox{ is an isolated zero of }
P|_{\Delta_1} \}.\]
and 
 \[\cQ_2:=  \left\{ q \in \Delta_{1,\sm} \colon \begin{array}{l}  P \mbox { and
} \Delta_1 \mbox{ have a common component } C  \\ 
\mbox{containing }q, \ord_C(P)=2 \mbox{ and } \ord_q(P)\geq 3.
\end{array}\right\}.\]
\end{proposition}

\begin{proof}
If all the partials of $g$ vanish at $p$ then, in particular, $\partial
g/\partial x$ and $\partial g/\partial y$ vanish, hence $p$ is a singular point
of  $\overline{\psi|_Y^{-1}\psi(p)}$ and $\psi(\Sigma)\subset \Delta_1$.
Moreover, if $p\in \Sigma$, then $p$ is the unique singular point of 
$\overline{\psi|_Y ^{-1}(\psi(p))}$.

For a general point $q$ on a  component $C$ of $\Delta$ one can find the
transversal type of the singularity along the corresponding component of
$\Sigma$ by Tate's algorithm. For more details we refer to \cite{MirEllThree}. 
We will use Tate's algorithm to identify the set of points where we do not have
a transversal surface singularity.

\textbf{$I_{\nu}$-fiber.} Suppose $C$ is a component of $\Delta$ of multiplicity
$\nu$ and $P\mid_C\not \equiv 0$.  We show now that  if $p\in \cP$ then
$q:=\psi(p)$ is either in $\Delta_{1,\sing}$ or $P(q)=0$ (i.e., $q\in \cQ_1$).

For each $q\in C$ we have that $\overline{\psi^{-1}(q)}$ has precisely one
singular point. Let $\Sigma'$ be the union of all these points. Let $t=0$ be an
equation for $C$ and let $s$ be a second local coordinate.

An easy calculation show that at a general point of $C$ the $x$-coordinate of
$p$ equals $-3Q(s,t)/2P(s,t)$.  As long as $P(s,t)\neq 0$ we can move the point
$x=-3Q(s,t)/2P(s,t),y=0$ to $(0,0)$. This yields a new local equation of $Y$,
namely
\[  8P^3 y^2=  8P^3 x^3 -36 PQ^2 x^2 + 2 P \Delta x - Q\Delta.\]
Since $\Delta(s,t)=t^\nu h(s,t)$, we have that $(Y,p)$ is equivalent to the
singularity 
\[ y^2=x^2 + t^\nu x +t^\nu\]
unless $h(t,s)P(t,s)Q(t,s)=0$. For degree reasons we can disregard $t^\nu x$, 
hence we have a transversal $A_{\nu-1}$ singularity unless
$h(t,s)P(t,s)Q(t,s)=0$. Since $\Delta=4P^3+27Q^2$ we have that then
$h(t,s)P(t,s)=0$.

\textbf{$I^*_{\nu}$-fiber, $\nu>0$.} Suppose $C$ is a component of $\Delta$ with
multiplicity $6+\nu$ and that $\ord_C(P)=2, \ord_C(Q)= 3$. Let $t=0$ be an
equation for $C$ and let $s$ be a second local coordinate. I.e., we can write
$P(s,t)=t^2P_1(s,t)$ and $Q(s,t)=t^3Q_1(s,t)$. 
As above, we move the point $(-3tQ_1(s,t)/P(s,t),0)$ to $(0,0)$.
Then we get a local equation of the form 
\[  8P_1(t)^3 y^2=  8P_1(t)^3 x^3 -36 t P_1(t)Q_1(t)^2 x^2 + 2t^2 P_1(t)
\Delta_2(t) x -t^3 Q_1(t)\Delta_2
(t).\]
Where $\Delta_2(t,s)=\Delta(t,s)/t^6$. Then
$\Delta_2=4P_1(t,s)^3+27Q_1(t,s)^2=t^\nu h(t,s)$ for some $h$. This local
equation is equivalent to a transversal $D_{4+\nu}$-singu\-la\-ri\-ty, unless
$P_1(t,s)Q_1(t,s)h(t,s)=0$. A reason similar to the $I_{\nu}$ case shows that 
either $p\in \Delta_{1,\sing}$ or $P_1$ and $Q_1$ vanish at $q$, which implies
that $P=t^2P_1$ vanish at least up to order 3 at $q$, i.e., $q\in \cQ_2$.

\textbf{Exceptional cases $II, III, IV, I_0^*, IV^*, III^*, II^*$.}

Of these we do only the most difficult cases $II^*, III^*$,  the other cases
being very similar.

Case $II^*$:
{}from Tate's algorithm it follows that we have a local equation of the form
\[ y^2=x^3+t^4 P_1(s,t) x + t^5Q_1(s,t) \]
such that $Q_1(s,t)$ does not vanish at a general point of $C$. Hence
$\Delta(s,t)=t^{10}(4t^2P_1(s,t)^3
+27Q_1(s,t)^5)$. This is a transversal $E_8$ singularity unless $Q_1(t,s)$
vanishes, but then  $q$ is a singular point of  
$\Delta_{1}$.

Case $III^*$: {}from Tate's algorithm it follows that we have a local equation
of
the form
\[ y^2=x^3+t^3 P_1(s,t)x+t^5Q_1(s,t) \]
such that $P_1(s,t)$ does not vanish at a general point of $C$. Hence
$\Delta(s,t)=t^{9}(4P_1(s,t)^3
+27tQ_1(s,t)^2)$. This is a transversal $E_7$ singularity unless $P_1(s,t)$
vanishes, but then $q$ is a singular point of  
$\Delta_{1}$.
\end{proof}

\begin{lemma}\label{lemDiscLocus} Suppose $q\in \PP^2$ is such that $P(q)=0$ and
$q$ is an isolated
double point of $\Delta$.
Then $\cP  \cap \psi^{-1}(q)=\emptyset$.
\end{lemma}
\begin{proof}
Using that $\Delta=4P^3+27Q^2$ and our assumptions on $\Delta$ and $P$ we obtain
that $Q=0$ is a smooth reduced curve in a neighborhood of $q$ and that $Q=0$
does not have a common component with $P=0$ or $\Delta=0$ in a neighborhood of
$p$. I.e., we have a local equation of the form
\[ y^2=x^3+ P x + s.\]
If $\Sigma$ and $\psi^{-1}(q)$ intersect, then the fiber needs to be singular
at that point, i.e., 
 $(x,y,t,s)=(0,0,0,0)$, However, it is easy to see that $Y$ is smooth at this
point, hence $\psi^{-1}(q)\cap \Sigma=\emptyset$.
\end{proof}

For a Weierstrass equation $g:=-y^2+x^3+Px+Q$ let  $\cQ:=(\Delta_{1,\sing}\cup
\cQ_1\cup \cQ_2)\setminus \cQ_3$, where $\cQ_1$ and $\cQ_2$ are defined as in
Proposition~\ref{prpDiscLocus} and 
\[ \cQ_3=\{q\in \Delta_{1,\sing} \colon P(q)=0 \mbox{ and } q \mbox{ is an
isolated double point of } \Delta \}.\]
Let 
\[  \cP':= \bigcup_{q \in \cQ} \overline{\psi|_Y^{-1}(q)_{\sing}} \subset Y. \]
Note that $\cP'$ is a finite set and contains the set $\cP$ of the previous
section. 

\begin{procedure}
Given an equation $y^2=x^3+Px+Q$ with homogeneous polynomials 
$P\in \CC[z_0,z_1,z_2]_{4n}, Q\in
\CC[z_0,z_1,z_2]_{6n}$ such that there is no $u \in \CC[z_0,z_1,z_2]\setminus
\CC$ with
$u^4|P$ and $u^6|Q$.
\begin{enumerate}
\item Set $Y=\{(x,y,z_0,z_1,z_2) \in \PP(2n,3n,1,1,1)\colon y^2=x^3+Px+Q\}$.
\item Determine the set $\cP'\subset Y$  defined above.
\item  
For each $p \in \cP'$  check whether $(Y,p)$ is contact equivalent to
a weighted homogeneous hypersurface singularity $(Y',p')$.

If not, then stop,
otherwise fix weights $w_{1,p},w_{2,p},w_{3,p},w_{4,p}$ and a weighted
homogeneous
polynomial $g_p\in \CC[y_1,y_2,y_3,y_4]$ such that $(Y,p)$ is contact equivalent
to
$(\{g_p=0\},0)$. Fix also a map $(Y,p)\to (\{g_p=0\},0)$. 
Let $d_p:=\deg g_p$,
$w_p:=\sum w_{i,p}$.

\item For each $p\in \cP'$ let $R(g_p)$ be the Jacobian ring of $g_p$. If
$(Y,{p})$ is an isolated singularity then set $\tilde{R}(g_p)=R(g_p)$. If
$(Y,{p})$ is not an isolated singularity then $\tilde{R}$ is defined as 
in Remark~\ref{remunite}.

\item Calculate the dimension $r_1$ of the cokernel of the natural map
\[ \CC[x,y,z_0,z_1,z_2]_{7n-3} \to \oplus_{p\in \cP'} \tilde{R}(g_p)_{2d_p-w_p}.
\]
\item Calculate the dimension $r_0$ of the cokernel of the natural map
\[ \CC[x,y,z_0,z_1,z_2]_{n-3} \to \oplus_{p\in \cP'} \tilde{R}(g_p)_{d_p-w_p}
.\]
\item If $r_0=0$ then $\rank\MW(\pi)=r_1$.
\item If $r_0>0$ then $\rank\MW(\pi)\leq r_1$.
\end{enumerate}
\end{procedure}

\begin{proof} As is shown above $\cP'$ is finite and contains $\cP$. For each
$p\in \cP'\setminus \cP$ we have that $(Y,p)$ is smooth or a transversal $ADE$
surface singularity. By \ref{prpTransSurf} it follows that $H^4_p(Y,\QQ)=0$.
Hence
to calculate the cokernel
of 
$H^4(U,\QQ)(1) \to \oplus_{q\in \cP} H^4_q(Y,\QQ)$, we can replace $\cP$ by
$\cP'$.

We proceed by calculating $h^{3,1}(H^4(Y,\CC))$ and $h^{2,2}(H^4(Y,\CC))$.
Combining
Proposition~\ref{propLocCoh} with Theorem~\ref{thmCoKer}
yields that
\begin{enumerate}
\item $ h^{3,1}(H^4(Y,\CC))\leq r_0$ and $h^{2,2}(H^4(Y,\CC))_{\prim} \leq
r_1$. 
\item If $r_0=0$ then $h^{3,1}(H^4(Y,\CC))=h^{4,0}(H^4(Y,\CC))=0$. Since
$H^4(Y,\QQ)$ has a
pure weight 4 Hodge structure it 
follows that $h^{1,3}(H^4(Y,\CC))=h^{0,4}(H^4(Y,\CC))=0$, hence
$H^4(Y,\CC)$ is of pure type $(2,2)$ and \[\rank
H^4(Y,\CC)_{\prim}\cap H^{2,2}(H^4(Y))_{\prim}=r_1.\]
\end{enumerate}
Applying Theorem~\ref{thmMW} finishes the proof.
\end{proof}

\begin{remark} An elliptic curve $E$ over $\CC(t_1)$ is for trivial reasons also
an elliptic curve over $\CC(t_1,t_2)$. We discuss what the outcome of our
method 
is, if we apply it to such $Y$. Note that $Y$ is
defined as the zero-set of 
\[ -y^2+x^3+P(z_0,z_1)x+Q(z_0,z_1) \]
i.e., $Y$ is a cone over an elliptic surface. Here we assume that $n$ is such
that $\deg(P)=4n$ and $\deg(Q)=6n$.
The discriminant curve is a union of lines through $(0:0:1)$. {}From this it
follows that $\cP'=\{(0:0:0:0:1)\}$. For simplicity assume that the
$(0:0:0:0:1)$ is an isolated singularity.

For $p=(0:0:0:0:1)$ we have a \emph{local} equation 
\begin{equation} \label{ESeqn} -v^2+u^3+P(s,t)u+Q(s,t)=0 \end{equation}
i.e., we have $d_p=6n$ and $w_p=5n+2$. 
Our algorithm tells us that we should calculate the dimension $r_1$ of the
cokernel of 
\[ \CC[x,y,z_0,z_1,z_2]_{7n-3} \to  \tilde{R}(g_p)_{7n-2}\]
and  calculate the dimension $r_0$ of the cokernel of 
\[ \CC[x,y,z_0,z_1,z_2]_{n-3} \to \oplus_{p\in \cP'} \tilde{R}(g_p)_{n-2} .\]
It is easy to see that both maps are the zero map. In particular, our method
tells
us that
\[ \rank \MW(\pi) \leq r_1 = \dim R(g_p)_{7n-2} = h^{1,1}(S)_{\prim} \]
where $S$ is the elliptic surface defined by (\ref{ESeqn}). Of course, we could
obtain this inequality directly, i.e., by applying the Shioda-Tate formula to
$S$.
\end{remark}

\part{Examples}

\section{Examples}\label{secEasy}

\begin{example}
Consider the elliptic threefold $Y$
\[ y^2+x^3+z_0^2z_2^2(z_0z_2-z_1^2).\]
The  locus $\Sigma$ of $Y$ is given by $y=w=z_0z_2=0$, i.e., is 1-dimensional.

The discriminant curve is $z_0^2z_2^2(z_0z_2-z_1^2)$. The set $\cP'$ consists of
three points $p_1=(0:0:1:0:0)$, $p_2=(0:0:0:1:0)$, $p_3=(0:0:0:0:1)$. Note that
$\Sigma$ is one dimensional in this case.

At $p_1$ and $p_3$ we have a local equation of the form
\[ v^2=u^3+t^2s^2+s^3\]
Set weights for $s,t,u,v$ as $2,1,2,3$.   Then this equation is weighted
homogeneous of degree $6$, and
\[ R(g_p)_{d_p-w_p}=0,  \; R(g_p)_{2d_p-w_p} = \spa
\{\overline{t^4},\overline{s^2},\overline{rt^2},\overline{rs}\}. \]
Along $v=u=s=0$ we have a transversal $A_2$-singularity. The Milnor algebra of
an
isolated $A_2$-singularity $v^2+u^3+t^2$ is generated by $1$ and $u$. If we
homogenize these two monomials we get $t^4$ and $ut^2$. Hence
\[ \tilde{R}(g_p)_{2d_p-w_p}=
R(g_p)_{2d_p-w_p}/(\overline{t^4},\overline{ut^2})=\spa \{
\overline{s^2},\overline{us}\}.\]
For $p=p_1$ we have that, after  homogenizing,  $s^2$ corresponds to
$z_0^2z_2^2$ and $xs$ corresponds to $xz_0z_2$. For $p=p_3$ we get similarly
that $\tilde{R}_{g_p}$ is generated by $\overline{z_0^2z_2^2}$ and
$\overline{xz_0z_2}$.

At $p=p_2$ we have a local equation of the form
\[ v^2=u^3+t^2s^2\]
If we set weights for $s,t,u,v$ as $  2,2,1,3$ we get a weighted homogeneous
equation of degree $12$. Again $R(g_p)_{d_p-w_p}=0$. We get that
$R(g_p)_{2d_p-w_p}$ is four dimensional, and that 
\[ \tilde{R}(g_p)_{2d_p-w_p} = 0.\]

This implies that $r_0=0$ and $r_1$ is the cokernel of
\[ \CC[x,y,z_0,z_1,z_2]_4 \to \tilde{R}(g_{p_1})_{4} \oplus
\tilde{R}(g_{p_3})_{4}. \]
Since both summands have the same generators it turns out that the cokernel has
dimension 2. In particular, $\rank \MW(\pi)$ is 2. The sections
$(x=\omega^i z_0z_2,y=z_0z_1z_2)$ for $i=0,1$ generate a finite-index subgroup
of $\MW(\pi)$. 

In order to determine the torsion subgroup of $\MW(\pi)$: fix a general line
$\ell$ in $\PP^2$ and consider
$\pi_\ell:\pi^{-1}(\ell) \to \ell$. Then $\pi^{-1}(\ell)$ is a rational elliptic
surface with $2IV$ fibers and $2II$ fibers. Such an elliptic surface has tivial
torsion subgroup \cite{ShOg}, hence $\MW(\pi)$ has no torsion.
\end{example}

\begin{example}
The second author has given several examples of elliptic threefolds with higher rank. 
For example if $Y$ is given by $y^2=x^3+f(z_0,z_1,z_2)$ where $f=0$ defines a sextic  in 
$\PP^2$ with $9$ cusps then the rank of $\MW(\pi)$ equals $6$. (See \cite{jconst,diff})
\end{example}

\section{An application}\label{secBeh}

The following construction of Calabi-Yau threefolds is due to F.~Hirzebruch and
was 
communicated to us by N.~Yui. Some of the details of the construction were
worked out in the 
Diplomarbeit 
\cite{NikoDiplom} of N.~Behrens.

\begin{construction}\label{ConHirz}
Let $S$ be a del Pezzo surface, i.e., the blow-up of $\PP^2$ in $m$ points
$p_1,\dots p_m$ in 
general position (meaning no three points on a line, and no six points on a
conic), $0\leq m
\leq 8$. By $E_i$ we denote the exceptional divisors of the blow-down morphism
$\varphi:
S\to \PP^2$. Let $L$ be the pullback to $S$ of a general line in $\PP^2$.

We consider the anti-canonical line bundle $\cL=\omega_S^{-1}=\cO(3L-\sum E_i)$
and define the rank 
$3$ bundle 
$\cE=\cO\oplus \cL^{-2} \oplus \cL^{-3}$.
Then $\PP(\cE)$ is a $\PP^2$-bundle over $S$. 
We use Grothendieck's definition of projective space, in particular 
$p_*\cO_{\PP(\cE)}(1)=\cE$ where $p$ is the bundle projection.
Fix sections \begin{eqnarray*}
X:=(0,1,0)& \in & H^0(\cL^{2} \oplus \cO \oplus \cL^{-1})=H^0(\str_{\PP(\cE)}(1)
\otimes \cL^2),\\
 Y:=(0,0,1) &\in& H^0(\cL^{3} \oplus \cL \oplus \cO)=H^0(\str_{\PP(\cE)}(1)
\otimes \cL^3),\\ 
 Z:=(1,0,0) &\in& H^0(\cO \oplus \cL^{-2}\oplus \cL^{-3})=H^0(\str_{\PP(\cE)}(1)
).\end{eqnarray*}
For general sections $g_2,g_3$ in $H^0(\cL^4)$ and $H^0(\cL^6)$ respectively,
the equation
\begin{equation}\label{HSinPB} Y^2Z=4X^3+g_2XZ^2+g_3Z^3\end{equation}
defines a \emph{smooth} hypersurface $W$ in $\PP(\cE)$. 
Note that $W$ is in the linear system defined by the anti-canonical line bundle 
$\omega_{\PP(\cE)}^{-1}=(p^*\cL^6) \otimes \cO_{\PP(\cE)}(3)$.
The projection onto $S$ 
defines an elliptic fibration $\pi:W\to S$ with a section.
\end{construction}

\begin{lemma}\label{lemCan} The threefold $W$ has trivial canonical bundle.
\end{lemma}
\begin{proof}
Since
\[ \omega_{\PP(\cE)}=p^*(\omega_S\otimes \det \cE) \otimes \str_{\PP(\cE))}(-3)=
 p^*\cL^{-6} \otimes \str_{\PP(\cE))}(-3)\]
and $\str_{\PP(\cE)}(W_7)=p^*\cL^{6} \otimes \str_{\PP(\cE))}(3)$ it follows
from the adjunction formula
that
\[\omega_{W_7}=\omega_{\PP(\cE)}(W_7)|_{W_7}=\str_{W_7}. \]
\end{proof}

In \cite{NikoDiplom} a detailed proof of the following result is given:
\begin{theorem}[{\cite[Theorem 2.35]{NikoDiplom}}] \label{NikoThm} Let $r=\rank
\MW(\pi)$. Then $W$ has the
following Hodge numbers:
\begin{enumerate}
\item $h^{1,0}(W)=h^{0,1}(W)=h^{2,0}(W)=h^{0,2}(W)=0$,
\item $h^{1,3}(W)=h^{3,1}(W)=0$,
\item $h^{0,3}(W)=h^{3,0}(W)=1$,
\item $h^{1,1}(W)=m+2+r$,
\item $h^{1,2}(W)=h^{2,1}(W)=272-29m+r$.
\end{enumerate}
The topological Euler characteristic $e(W) = -540+60m$.
\end{theorem}

\begin{remark} The fact that $h^{1,0}(W)=h^{2,0}(W)=0$ and that
$\omega_{W_7}=\str_{W_7}$ implies that $W_7$ is a Calabi-Yau threefold. For
Calabi-Yau threefolds finding their mirror partner 
is of particular interest. The line bundle $(p^*\cL^6) \otimes
\cO_{\PP(\cE)}(3)$ is not an
ample line bundle. (This follows e.g., since 
$\pi_*(\str_{\PP(\cE)}(1) \otimes \cL^{2})=\cE\otimes \cL^{2} = \cL^2\oplus
\str_S\oplus \cL^{-1}$.) Hence we are not in a position where Batyrev's mirror
construction \cite{BatMir} can be applied directly. In order to find a 
mirror family it is first of all necessary to compute the Hodge numbers of $W$.
This was the motivation behind \cite{NikoDiplom}.
\end{remark}

To actually find the Hodge numbers we need to determine the rank of $\MW(\pi)$.
In \cite{NikoDiplom} it is conjectured that $r=0$ for all such $W$. We apply our
methods to prove this conjecture. We  first calculate the Mordell-Weil rank by
computing $h^4(Y)$.  In the second half of this section we illustrate our
methods by determining 
all Hodge numbers by going through the various constructions, thus avoiding
a direct reference to Theorem~\ref{NikoThm}.

We know that $W$ is birational to a hypersurface $Y$ of degree $6n$ in some 
weighted projective space
$\PP(2n,3n,1,1,1)$. For $n=1,2$ such a threefold is a deformation of a rational
variety. Since $W$ is a Calabi-Yau hypersurface we have $n\geq 3$.

\begin{lemma}\label{lemMiln} There exists a degree 18 hypersurface $Y$ in
$\PP(6,9,1,1,1)$, birational to $W$ and such that $Y_{\sing}$ consists of
$(1:1:0:0:0)$ and $m$ isolated semi-weighted homogeneous  hypersurface 
singularities with Milnor number 50.  For each of these singularities we have
that $H^4_p(Y,\QQ)\cong \QQ(-2)^8$.
\end{lemma}

\begin{proof}
We need to consider $g_2, g_3$ as functions on $\PP^2$, rather than elements in
$H^0(S,\cL^i)$. Since $\varphi_* \cL= \cO(3) \otimes \cI_{p_1,\dots,p_m}$, it
follows that $g_2\in H^0(\cO(12) \otimes \cI_{p_1,\dots,p_m}^4)$ and $g_3\in
H^0(\cO(18) \otimes \cI_{p_1,\dots,p_m}^6)$. Let $P$ and $Q$ be the associated
weighted homogeneous polynomials of degree 12 and 18 respectively. Then
\begin{equation}\label{eqnWe} y^2=x^3+Px+Q\end{equation}
defines a degree 18 hypersurface $Y$ in $\PP(6,9,1,1,1)$ birational to $W$. 

Let $\tilde{\psi}:\PP \to \PP^2$ be the projection from $\{z_0=z_1=z_2=0\}$ to
the
plane $\{x=y=0\}$. 
Then $\psi
=\tilde{\psi}|_Y$ corresponds to the elliptic fibration on $W$. Note that $p$ is
defined on $Y\setminus \{(1:1:0:0:0)\}$.
Since $W$ is smooth all singularities (besides  $(1:1:0:0:0)$) lie in
$\psi^{-1}(p_i)$ for $i=1,\dots m$.  

Equation (\ref{eqnWe})  shows that $\overline{\psi^{-1}(p_i)}$  has equation
$Y^2Z=X^3+P(p_i)XZ^2+Q(p_i)Z^3$. In particular, $\overline{\psi^{-1}(p_i)}$ is
an
irreducible and reduced cubic plane curve and it  has at most one singularity.
Since $Y$ is singular at $q_i=(0:0:p_i)$, the same holds for
$\overline{\psi^{-1}(p_i)}$, and there are no other singular points  on
$Y\setminus
\{(1:1:0:0:0)\}$.

We proceed by calculating the Milnor number of $(Y,q_i)$. A local equation for
$Y$ around $q_i$ is 
\[v^2=4u^3+h_4(t,s)u+h_6(t,s)+ h.o.t.\]
An easy calculation, using that $W$ is smooth,
shows that the lowest degree part
\[v^2=4u^3+h_4(t,s)u+h_6(t,s)\]
defines a quasismooth surface in $\PP(2,3,1,1)$.  
In particular, $(Y,q_i)$ is  a semi-weighted homogeneous hypersurface
singularity, i.e., we may ignore the higher order terms.

To calculate the Milnor number of $(Y,q_i)$ we need to consider the Jacobian
ring $R$ of the
defining equation of the singularity. 
Using Lemma~\ref{lemdimJac} (proven below) it follows that 
\[ \sum \dim R_d t^d = 1+2t+4t^2+6t^3+8t^4+8t^5+8t^6+6t^7+4t^8+2t^9+t^{10}.\] 
Hence $\mu=\dim R=50$.

To calculate the local cohomology it suffices to determine $\dim R_{d-w}=R_{-1}$
and $\dim R_{2d-w}=\dim R_{5}$. The former space is 0, the latter space is
8-dimensional. Now apply Proposition~\ref{prpIsol} and Theorem~\ref{thmAdj}.
\end{proof}

\begin{lemma}\label{lemdimJac}
 Let $f\in \CC[x_0,\dots,x_{n+1}]$ be a weighted homogeneous polynomial of
degree
$d$ with weights $w_0,\dots,w_{n+1}$. Assume that each $w_i$ divides $d$ and
that $f=0$ has at most an isolated singularity at the
origin. Let $R$ be the Jacobian ring of $f$. Then
\[ \sum_k \dim R_k t^k = \prod \frac{t^{d-w_i}-1}{t^{w_i}-1}.\]
\end{lemma}

\begin{proof}
Since $f=0$ has at most a singularity at the origin it follows that the partials
of $f$ form a regular sequence in $\CC[x_0,\dots,x_{n+1}]$.
This implies that $R$ is resolved by its Koszul complex. An easy calculation
yields the proof.
\end{proof}

For the rest of this section, let $Y$ be the  degree $6n$ hypersurface  in
$\PP(6,9,1,1,1)$ constructed in the proof above. In particular,  $Y\cap
\{z_1=z_2=z_3=0\}=\{(1:1:0:0:0)\}$.  Let $q_i=(0:0:p_i)$.

The form of the singularity $(Y,q_i)$ allows us to use Dimca's results. For this
we first prove
the following two lemmas.

\begin{lemma}\label{lemqs} Let $T\subset \PP(6,9,1,1,1)$ be a quasismooth
hypersurface of degree $18$. Then $h^3(T)=546$ and the topological Euler
characteristic $e(T)=-542$.
\end{lemma}

\begin{proof} Since the topology of quasismooth hypersurfaces is invariant under
deformation, it suffices to prove this statement for $T$ given by
\[f:= y^2+x^3+z_0^{18}+z_1^{18}+z_2^{18}.\]
Let $R$ be the Jacobian ring of $f$. Using Griffiths-Steenbrink (see
Section~\ref{secCoh}) we know that
\[ h^3(T)=\dim R_0+\dim R_{18}+\dim R_{36}+\dim R_{54}.\]
An easy calculation shows that $\dim R_0=\dim R_{54}=1$ and $\dim R_{18}=\dim
R_{36}=272$. Hence $h^3(T)=546$. {}From Lefschetz' hyperplane theorem
(Proposition~\ref{prpLHT}) it follows that $h^i(T)=1$ for $i=0,2,4,6$ and all
other Betti numbers vanish. {}From this the equality $e(T)=4-546=-542$ follows.
\end{proof}

\begin{lemma} The topological Euler characteristic $e(Y)$ of $Y$ equals
$-542+50m$.
\end{lemma}
\begin{proof}Let $T$ be a quasismooth hypersurface of the same degree of $Y$.
{}From e.g. \cite[Corollary 5.4.4]{Dim} it follows that
\[e(Y)=e(T)+\mu\]
where $\mu$ is the total Milnor number of $Y$, i.e., the sum of the Milnor
numbers of the singularities of $Y$ besides $(1:1:0:0:0)$.
{}From Lemma~\ref{lemMiln} and Lemma~\ref{lemqs} it follows that
$e(Y)=-542+50m$.
\end{proof}

Using the Lefschetz hyperplane theorem (Proposition~\ref{prpLHT})  we obtain
that 
\[ h^0(Y)=h^2(Y)=h^6(Y)=1 \mbox{ and } h^1(Y)=h^5(Y)=0.\]
Hence $h^3(Y)= 546-50m+h^4(Y)-1$.

To calculate $h^4(Y)$ we use Dimca's method. For this we need some results on
linear systems on $\PP^2$.

\begin{definition}  Let $L_d(k^m)$ be the linear system of degree $d$ curves
having a point of order $k$  at $p_1,\dots,p_m$. The defect of $L_d(k^m)$ equals
$m\frac{k(k+1)}{2}- \codim_{\CC[z_0,z_1,z_2]_d} L_d(k^m)$, i.e., the difference
between the expected codimension and the actual codimension.  
\end{definition}

We are interested in $L_{18}(6^m)$ and $L_{12}(4^m)$, in the case that the $m$
points are the $p_i$.

\begin{proposition}\label{propDef} For $k>0$ we have that the linear system
$L_{3k}(k^m)$ has no defect. 
\end{proposition}
\begin{proof}
Note that $L_{3k}(k^m)$ is isomorphic to  $H^0(S,\str_S(3kH-k\sum E_m))$.
Set $D=3H-\sum E_i$ and let $C$ be an irreducible smooth curve in $|D|$. (Such a
curve exists since the $p_i$ are in general position and $m\leq 8$.) Since $C$
is the strict
transform of a degree 3 curve in $\PP^2$ we have that $g(C)=1$. 

Let $\cL=\str(D)|_C$. Then $\deg(\cL)=D^2=9-m>0$. Using $g(C)=1$ we find for
$t>0$ that $h^0(\cL^t)=t(9-m)$ and $h^1(\cL^{\otimes t})=0$. 

Consider now the long exact sequence in cohomology associated to 
\[ 0\to \str_S((t-1)D)\to \str_S(tD)\to \cL^{\otimes t}\to 0.\]
Since for $t\geq 1$ we have that  $h^1(\cL^{\otimes t})=0$, we find that $h^1(
\str_S(tD) ) \leq h^1( \str_S((t-1)D))$. 
Note that for $t=1$ we have that $h^1(\str_S((t-1)D))=h^{0,1}(S)=0$. Combining
this yields that $h^1( \str_S(tD) )=0$ for $t\geq 0$.
This implies that
\[h^0(\str_S(tD))=h^0(\str_S((t-1)D))+h^0(\cL^{\otimes t})=
h^0(\str_S((t-1)D))+t(9-m)\]
whence 
\[h^0(\str_S(tD)) = \frac{t(t+1)(9-m)}{2}+h^0(\str_S)=\frac{t(t+1)(9-m)}{2}+1.\]
The expected dimension of $L_{3k}(k^m)$ equals 
\[ \frac{(3k+1)(3k+2)}{2}-m \frac{k(k+1)}{2}=\frac{k(k+1)(9-m)}{2}+1.\]
This implies that $L_{3k}(k^m)$ has the expected dimension and thus
$L_{3k}(k^m)$ has
no defect.
\end{proof}

\begin{proposition}\label{hodgeY} We have that $h^4(Y)=1$, hence
$h^3(Y)=546-50m$.
\end{proposition}

\begin{proof}

{}From Dimca's work, (the dimension zero case of  Sections~\ref{secLoc}
and~\ref{secGlue}), it follows that the primitive cohomology
$H^4(Y,\QQ)_{\prim}$ is
isomorphic to the cokernel of
\[H^4(\PP\setminus Y,\QQ) \to \oplus_{q_i} H^4_{q_i}(Y,\QQ).\]
{}From Lemma~\ref{lemMiln} we know that $H^4_{q_i}(Y,\QQ)=\QQ(-2)^8$.

A local equation of $(Y,q_i)$ (see the proof of Lemma~\ref{lemMiln}) is
\[f_{q_i}:=-v^2+4u^3+h_{4,i}(t,s)u+h_{6,i}(t,s).\]
This equation is weighted homogeneous. Moreover, we know that this is an
equation
of a quasismooth surface. Let $R(f_{q_i})$ denote the Jacobian ring of
$f_{q_i}$.

{}From Proposition~\ref{propLocCoh} and Theorem~\ref{thmCoKer} it follows that
the
cokernel of $H^4(\PP\setminus Y,\CC) \to \oplus H^4_{q_i}(Y,\CC)$ equals the
cokernel of 
$ \Gr_P^2 H^4(\PP \setminus Y,\CC) \to \oplus H^4_{q_i}(Y,\CC)$. Using the
natural
maps 
\[ \CC[z_0,z_1,z_2]_{12}x \oplus \CC[z_0,z_1,z_2]_{18}\twoheadrightarrow
R(f)_{18}\twoheadrightarrow  \Gr_P^2 H^4(\PP\setminus Y,\CC)\]
it follows that it suffices to prove that
\begin{equation}\label{tayeqn} \CC[z_0,z_1,z_2]_{12}x \oplus
\CC[z_0,z_1,z_2]_{18}
\to \oplus H^4_{q_i}(Y,\CC) =\oplus_i R(f_{q_i})_{5} \end{equation}
is surjective.

 Define $T_{q,m,d}: \CC[z_0,z_1,z_2]_d \to \CC^{m(m+1)/2}$ to be the $(m-1)$st
part of the Taylor
expansion  around $(\alpha_1,\alpha_2,\alpha_3)$ for some
fixed lift of $q\in \PP^2$ to $\CC^3$. Then the map form (\ref{tayeqn}) can be
factored as 
\[    \CC[z_0,z_1,z_2]_{12}x \oplus \CC[z_0,z_1,z_2]_{18} \stackrel{\oplus
(T_{q_i,4,12} \oplus T_{q_i,6,18})}{\longrightarrow}  \oplus_i
\left(\CC^{10}\oplus \CC^{21}\right) \to \oplus R(f_{q_i})_{5}.\]
The first map is surjective by Proposition~\ref{propDef} and the second map is
surjective since it is a projection. {}From this the lemma follows.\end{proof}

Applying Theorem~\ref{thmMW} yields:
\begin{corollary} We have $\rank \MW(\pi)=0$.
\end{corollary}

\begin{remark}
Actually, $\MW(\pi)=0$: let $\ell\subset \PP^2$ be a general line. Then
$\pi_\ell:\pi^{-1}(\ell)\to \ell$ is an elliptic surface with $36$ $I_1$ fibers.
(This follows from the fact that  the discriminant curve is reduced.)
Suppose $\MW(\pi_{\ell})$ has a torsion section of order $k$, then one can
factor the $j$-map over $X_1(k)\to X(1)$    since this map is ramified at
$\infty$ with ramifaction index $k$ it turns out that $\pi_\ell$ has a fiber of
type $I_{km}$ of $I^*_{km}$ for some $m\geq 1$. Since all fibers of $\pi_\ell$
are of type $I_1$ it follows that $\MW(\pi_\ell)$ has trivial torsion part,
hence $\MW(\pi)$ has trivial torsion.
\end{remark}

\end{document}